\documentclass[11pt]{article}
\usepackage{amsmath, amsfonts, amsthm, amssymb, xypic}
\usepackage{graphicx}
\usepackage{float}
\usepackage{color}
\usepackage{verbatim}
\usepackage{enumerate}

\newtheorem{thm}{Theorem}[section]
\newtheorem{lma}[thm]{Lemma}
\newtheorem{cor}[thm]{Corollary}
\newtheorem{defn}[thm]{Definition}

\newtheorem{prop}[thm]{Proposition}

\newtheorem{rem}[thm]{Remark}

\newtheorem{eg}[thm]{Example}

\newcommand{\hd}{\dim_\textup{H}}
\newcommand{\bd}{\dim_\textup{B}}
\newcommand{\ubd}{\overline{\dim}_\textup{B}}
\newcommand{\lbd}{\underline{\dim}_\textup{B}}
\newcommand{\qad}{\dim_\textup{qA}}
\newcommand{\ad}{\dim_\textup{A}}

\newcommand{\eps}{\varepsilon}
\newcommand{\R}{\mathbb{R}}

\newcommand{\N}{\mathbb{N}}

\providecommand{\I}{\mathcal{I}}
\newcommand{\ii}{\textbf{\emph{i}}}
\newcommand{\jj}{\textbf{\emph{j}}}
\renewcommand{\epsilon}{\varepsilon}

\title{The box dimensions  of exceptional \\ self-affine sets in $\mathbb{R}^3$}

\author{Jonathan M. Fraser \& Natalia Jurga}

\begin{document}

\date{}

\maketitle

\begin{abstract}
We study the box dimensions  of   self-affine sets in $\mathbb{R}^3$ which are generated by a finite collection of generalised permutation matrices.  We obtain bounds for the dimensions which hold with very minimal assumptions and give rise to sharp results in many cases.  There are many issues in extending the well-established planar theory to  $\mathbb{R}^3$ including that the principal planar projections are (affine distortions of) self-affine sets with overlaps (rather than self-similar sets)  and that the natural modified singular value function fails to be sub-multiplicative in general.  We introduce several new techniques to deal with these issues and hopefully provide some insight into the challenges in extending the theory further.
\\ \\
\emph{Mathematics Subject Classification} 2010:  primary: 28A80, 37C45, secondary: 15A18.
\\ \\
\emph{Key words and phrases}: self-affine set, box  dimension,  singular values.
\end{abstract}

\section{Introduction} \label{intro}

Due to two recent breakthroughs in the dimension theory of fractals, it is particularly timely to gain a better understanding of exceptional self-affine sets, especially in $\mathbb{R}^3$:

\begin{enumerate}
\item Falconer's seminal work from the 1980s established formulae for the Hausdorff and box dimensions of arbitrary self-affine sets which held `almost surely' in a natural sense, see \cite{falconer1,falconer2}.  However, it has emerged in the last few years that these generic formulae actually hold \emph{surely} outside of a very small family of exceptions, see \cite{barany, hochman}.  These exceptions include the systems we consider here: IFSs generated by generalised permutation matrices.
\item `Carpet like' self-affine sets in $\mathbb{R}^3$ (which are the special case of the systems we consider where all the matrices are diagonal) have recently been used to provide counterexamples to a famous open conjecture in the dimension theory of dynamical systems, thus indicating the subtlety of this class of set.  In particular, see \cite{das} where a self-affine set in $\mathbb{R}^3$ is constructed which does not have an invariant measure of maximal Hausdorff dimension, see \cite{kenyon}.
\end{enumerate}

The main aim of this paper is to investigate the box dimensions of self-affine sets generated by generalised permutation matrices, which we will call \emph{self-affine sponges}. There are two major obstacles that arise in the 3-dimensional setting and not in the 2-dimensional setting:

\begin{enumerate}
\item The natural (modified) singular value function is not generally sub- or super- multiplicative. This prevents standard methods for: (a) proving that the pressure exists going through and (b) comparing the pressure with the value that the (modified) singular value function takes at points in the natural partition for the box dimension, which is necessary for obtaining lower bounds on the box dimension.
\item The natural formulae for the box dimensions of exceptional self-affine sets in any dimension rely on knowledge of the projections onto principal subspaces.  In the 2-dimensional setting the principal 1-dimensional projections are (graph-directed) self-similar sets, but in the 3-dimensional setting the  principal 2-dimensional projections are more difficult to handle for two reasons.  First, they are overlapping (graph-directed)  self-affine sets and it is not known if the box dimensions even exist, and secondly, because of the affine distortion, one needs precise control on the covering number of not just the principal 2-dimensional projections, but also certain affine images thereof.  
\end{enumerate}

In this paper we provide two new techniques, which partially deal with the issues described above:

\begin{enumerate}
\item In order to deal with the lack of multiplicativity we pass to an appropriately chosen subsystem which both `carries' the pressure and dimension and satisfies the required multiplicativity.
\item In order to deal with the affine distortion in the principal 2-dimensional projections we introduce a new covering technique for cylinders based on the quasi-Assouad dimension of the principal 1-dimensional projections, see Lemma \ref{keycover1}. 
\end{enumerate}

Our main results are bounds for the box dimensions of self-affine sponges which arise from natural pressure functions.  In certain cases we obtain precise results, in particular if the group generated by the induced permutations on the coordinate axes  is $S_3$.  We also provide a surprising example in the case where the rotational components are trivial (the defining matrices are diagonal) where the upper box dimension is strictly smaller than the value predicted by extending the planar theory to $\mathbb{R}^3$.  Additionally, we provide a simple example of a planar non-overlapping self-affine set of the type studied in \cite{fraser} whose box dimension is \emph{not} continuous in the linear parts of the defining affinities. This contrasts with the result of Feng and Shmerkin \cite{fengshmerkin} which indicates that `generically' the dimension of a self-affine set \emph{is} continuous in the linear parts of the defining affinities.

The remainder of the paper is structured as follows. In section \ref{prel} we provide some preliminaries. In Section \ref{sponges} we introduce various classes of self-affine sponges. In Section \ref{dimformula} we introduce the modified singular value functions which are natural for our setting and prove the existence of the associated pressure functions. Section \ref{results} contains the main results of the paper and Sections \ref{notation} to \ref{dimdrop} contain proofs of these results. In Section \ref{discont} we provide an example of a planar system whose box dimension is discontinuous in the linear parts of the defining affinities. 

\section{Preliminaries} \label{prel}

Given a subset $F \subset \R^d$ and $\delta>0$, let $N_{\delta}(F)$ denote the minimum number of $d$-dimensional boxes of sidelength $\delta$ needed to cover $F$. The upper box dimension of $F$ is then defined as
$$\ubd F= \limsup_{ \delta \to 0} \frac{\log N_{\delta}(F)}{-\log \delta}$$
and the lower box dimension of $F$ is defined as
$$\lbd F= \liminf_{\delta \to 0}  \frac{\log N_{\delta}(F)}{-\log \delta}.$$
If the limits in these expressions exist we call the common value $\bd F= \ubd F=\lbd F$ the box dimension of $F$. We also let $\hd F$ denote the Hausdorff dimension of $F$.

Let $\{S_i\}_{i \in \mathcal{I}}$ be a finite collection of affine contracting maps $S_i : \R^d \to \R^d$, which we call an iterated function system (IFS). It is well-known that there exists a unique, non-empty, compact set $F$ such that
$$F=\bigcup_{i \in \mathcal{I}} S_i(F).$$
We say that $F$ is a self-affine set and call it the attractor of $\{S_i\}_{i \in \mathcal{I}}$.  To avoid trivialities we assume $\mathcal{I}$ consists of at least two maps.

Let $\mathcal{I}^*=\bigcup_{n=1}^{\infty} \mathcal{I}^n$ denote the set of all finite words over $\mathcal{I}$. For $\ii \in \I^*$ let $|\ii|$ denote the length of the word $\ii$. For $n \in \N$ and $\ii \in \I^*$ with $|\ii| >n$ let $\ii|_n$ denote the truncation of $\ii$ to its first $n$ digits. Given $\ii=i_1 \ldots i_n \in \mathcal{I}^n$, let $S_{\ii}:=S_{i_1} \circ \cdots \circ S_{i_n}$. For $1 \leq j \leq d$, let $\alpha_j(\ii)$ denote the $j$th singular value of the linear part of $S_{\ii}$. We also denote $\alpha_{min}:= \min\{\alpha_d(i)\}_{i \in \mathcal{I}}$ and $\alpha_{max}:= \max\{\alpha_1(i)\}_{i \in \mathcal{I}}$. For $s \geq 0$, define $\phi^s: \mathcal{I}^* \to \R$ as
\begin{equation*}
\phi^s(\ii)= \left\{ 
\begin{array}{ccc}
\alpha_1(\ii) \cdots \alpha_{\lfloor s \rfloor}(\ii) \alpha_{\lfloor s \rfloor +1}(\ii)^{s-\lfloor s \rfloor}  &\textnormal{if}&  0 \leq s \leq d \\
(\alpha_1(\ii) \cdots \alpha_d(\ii))^{\frac{s}{d}}&\textnormal{if}& s \geq d
\end{array} \right. 
\end{equation*}
In \cite{falconer}, Falconer showed if the Lipschitz constant of each map $S_i$ is less than $1/2$ (this includes an improvement of Solomyak \cite{solomyak}) then for Lebesgue typical translations, $\hd F= \bd F= \min\{d, s_0\}$ where $s_0$ was defined as the unique solution to
$$\lim_{n \to \infty} \left( \sum_{\ii \in \mathcal{I}^n} \phi^{s_0}(\ii) \right)^{\frac{1}{n}}=1,$$
where the limit above exists by submultiplicativity of $\phi^s$. We call $\min\{d, s_0\}$ the affinity dimension of $F$. In other words, the affinity dimension is the almost sure (in terms of the translations) value for the Hausdorff and box dimensions of a self-affine set.

However, it has since transpired that in the planar setting, the affinity dimension is in fact the sure value of the Hausdorff and box dimensions of a self-affine set, outside of a very small family of exceptions. In particular, in recent results of B\'ar\'any, Hochman and Rapaport \cite{barany} followed by Hochman and Rapaport \cite{hochman} it has been shown in the planar case that if the collection of linear parts of $\{S_i\}_{i \in \mathcal{I}}$ are totally irreducible, meaning that there is no finite collection of lines which are left invariant under the linear parts, and additionally the IFS satisfies an exponential separation condition, then $\hd F= \bd F= \min\{2, s_0\}$. An analogous result is expected to hold in $\R^3$.

The self-affine sets we will consider are not totally irreducible and belong to the family of exceptions for which the generic formula does not hold.  Similar classes of exceptional self-affine sets are very well-studied, especially in the plane.  The simplest and most fundamental class of examples is the family of self-affine carpets introduced by Bedford and McMullen in the 1980s \cite{bedford, mcmullen}.  Here the affine maps act on the plane and   share a common linear part, which is a diagonal matrix. Bedford and McMullen computed the box and Hausdorff dimensions of these sets and notably these can be distinct and strictly smaller than the affinity dimension.  Since then various other classes have been considered with increasing complexity. In the  Lalley-Gatzouras family \cite{lalley-gatz}   the linear parts are allowed to vary but an alignment condition (where the eigenspace corresponding to the largest eigenvalue is common)  is preserved.  In the Bara\'nski family \cite{baranski} the alignment condition is relaxed but a grid structure is preserved forcing the principal projections to be self-similar sets satisfying the open set condition.  In the Feng-Wang family \cite{fengaffine} the grid structure is removed.  Finally, in the family considered by Fraser \cite{fraser} the matrices can be both diagonal and anti-diagonal.  In particular, this family is irreducible (no subspaces are preserved)  but not totally irreducible (the union of the two principal subspaces \emph{is} preserved) and   is the 2-dimensional analogue of insisting that the matrices are generalised permutation matrices, which is the class we consider in this paper in $\mathbb{R}^3$.

The theory of exceptional ``carpet like'' self-affine sets is unsurprisingly  less well-developed in $\mathbb{R}^3$.  Kenyon and Peres \cite{kenyon} considered the analogue of the Bedford-McMullen family in $\mathbb{R}^d$ and computed the box and Hausdorff dimensions, but there are very few results outside of this. Generally the word ``sponge'' is used instead of ``carpet'' in dimensions at least 3.

\section{Self-affine sponges} \label{sponges}

Let $\{S_i\}_{i \in \mathcal{I}}$ be a finite set of affine contractions on $\mathbb{R}^3$ such that the linear part of each $S_i$ is a  generalised permutation matrix. Without loss of generality we assume that $S_i([0,1]^3) \subset [0,1]^3$ for each $i \in \mathcal{I}$. We will call its attractor $F$ a \emph{self-affine sponge}. These will be the objects of study throughout the rest of this paper.

Note that the linear part of each map $S_i$ corresponds to some permutation of the three axes. Let $G$ be the group generated by these permutations. 

Observe that for each $\ii \in \mathcal{I}^*$, there exists a permutation $\sigma$ of $\{1,2,3\}$ such that $r_{\ii}^{\sigma(1)} \geq r_{\ii}^{\sigma(2)} \geq r_{\ii}^{\sigma(3)}$, where $r_{\ii}^j$ denotes the contraction seen by the $j$th direction under $S_{\ii}$. We call this an ordering of the map $S_{\ii}$ and note a map may admit multiple orderings since $\sigma$ may not be unique (if the contraction in two directions is equal). Given $\ii \in \mathcal{I}^*$, we call $S_{\ii}([0,1]^3)$ a cylinder (in the construction of $F$). We define the ordering of the cylinder $S_{\ii}([0,1]^3)$ to be the ordering of the map $S_{\ii}$. Thus a cylinder may also have multiple orderings. We say that the orderings of two maps or cylinders agree if they both have an ordering which coincide, and we say that they disagree if they do not possess orderings which coincide.

We will be particularly interested in the following classes of self-affine sponge.

\begin{eg}[$S_3$-sponges]
Suppose $G \cong S_3$, where $S_3$ denotes the symmetric group of order 3. In this setting, the IFS is irreducible, meaning that no one-dimensional or two-dimensional linear subspace of $\R^3$ is preserved by the linear parts of the maps in the IFS. We call the attractor of this IFS an $S_3$-sponge.   \label{eg1}
\end{eg}

One should think of $S_3$-sponges as the $3$-dimensional analogue of the planar carpets considered by Fraser \cite{fraser} when there is at least one diagonal and anti-diagonal matrix.

\begin{eg}[Ordered sponges] \label{eg2}
Suppose the linear part of each map in the IFS is a diagonal matrix
$$A_i=\begin{pmatrix}a^1_i & 0&0 \\0&a^2_i&0\\0&0&a^3_i \end{pmatrix}.$$
We also assume that the set of matrices satisfies the co-ordinate ordering condition: there exists a permutation $\sigma$ of $\{1,2,3\}$ such that for all $i \in \mathcal{I}$, $|a_i^{\sigma(1)}|\geq |a_i^{\sigma(2)}| \geq |a_i^{\sigma(3)}|$. We call this an ordered IFS and we call the attractor an ordered sponge.
\end{eg}

One should think of ordered sponges as the $3$-dimensional analogue of the planar carpets considered by Lalley-Gatzouras  \cite{lalley-gatz} but without requirement to have a column structure.  There is also an intermediate case between the $S_3$ and ordered cases which we can handle.

\begin{eg}[$S_2$-partially ordered sponges] \label{eg2.5}
Suppose the linear part of each map in the IFS is either of the form
$$A_i=\begin{pmatrix}a^1_i & 0&0 \\0&a^2_i&0\\0&0&a^3_i \end{pmatrix} \qquad \text{or} \qquad A_i=\begin{pmatrix}0&a^1_i & 0 \\a^2_i&0&0\\0&0&a^3_i \end{pmatrix} $$
and there is at least one of each type.  Moreover, we assume that for all $i \in \mathcal{I}$,
\[
\min\{|a_i^{1}| , |a_i^{2}| \}\geq |a_i^{3}|.
\]
We call this an $S_2$-partially ordered IFS and we call the attractor an $S_2$-partially ordered sponge. (Note that $G \cong S_2$ in this case, and that the ordering is only partial).
\end{eg}

\begin{eg}[Ordered and separated sponges] \label{eg3}
Suppose we have an ordered IFS, so the maps are given by
$$S_i(\cdot)=\begin{pmatrix}a^1_i & 0&0 \\0&a^2_i&0\\0&0&a^3_i \end{pmatrix}(\cdot)+\begin{pmatrix} t_i^1 \\ t_i^2 \\t_i^3\end{pmatrix}$$
and $\sigma$ is the permutation from example \ref{eg2}. Let $\{T_i\}_{i \in \mathcal{I}}$ be the planar IFS associated to the maps
$$T_i(\cdot)=\begin{pmatrix}a^{\sigma(1)}_i & 0 \\0&a^{\sigma(2)}_i\end{pmatrix}(\cdot)+\begin{pmatrix} t_i^{\sigma(1)} \\ t_i^{\sigma(2)} \end{pmatrix}.$$
and $\{U_i\}_{i \in \mathcal{I}}$ be the  IFS acting on the line associated to the maps $U_i(\cdot)= a^{\sigma(1)}_i  (\cdot)+  t_i^{\sigma(1)} $.
Suppose that for all $i, j \in \mathcal{I}$, either $T_i((0,1)^2) \cap T_j((0,1)^2)= \emptyset$ or $T_i((0,1)^2) =  T_j((0,1)^2)$, and similarly, either $U_i((0,1)) \cap U_j((0,1))= \emptyset$ or $U_i((0,1)) = U_j((0,1))$.

  In this case we say that $\{S_i\}_{i \in \mathcal{I}}$ is ordered and separated and call its attractor an ordered and separated sponge.
\end{eg}

One should think of ordered and separated sponges as the $3$-dimensional analogue of the planar carpets considered by Lalley-Gatzouras  \cite{lalley-gatz}, although the ordered and separated family is more general since non-trivial reflections are permitted.

It is natural and expected that the projections of $F$ onto the principal 1- and 2-dimensional subspaces will play a key role in the dimension theory of $F$.  However, in the $\mathbb{R}^3$ case these projections are more difficult to handle than in the planar case.  Before introducing the key quantities, we need to briefly recall the quasi-Assouad dimension.  This is a notion of dimension introduced   by L\"u and Xi \cite{LuXi}, and is  defined by $\qad  E= \lim_{\eps \to 0} h_E(\eps)$ where, for $\eps \in (0,1)$,
\begin{eqnarray*}
h_E(\eps)\ = \    \inf \Bigg\{ \  \alpha &:& \text{     there exists a constant $C >0$ such that,} \\
&\,& \hspace{3mm}  \text{for all $0<r\leq R^{1+\eps}<1$ and $x \in E$ we have } \\ 
&\,&\hspace{24mm}  \text{$ N_r\big( B(x,R) \cap E \big) \ \leq \ C \bigg(\frac{R}{r}\bigg)^\alpha$ } \Bigg\}.
\end{eqnarray*}
The quasi-Assouad dimension leaves an `exponential gap' between the scales $r$ and $R$, which in some settings makes it rather different from the more well-known  Assouad dimension, $\ad$.  In general one has $\ubd E \leq \qad E \leq \ad E$ for bounded subsets of $\mathbb{R}^d$.

\begin{defn}[Principal one-dimensional projections]
Let $\ii \in \mathcal{I}^*$. We define:
\begin{enumerate}[(a)]
\item  $\pi_\ii^1$ to be the projection  onto the 1-dimensional subspace parallel to the longest edges of the cuboid $S_\ii([0,1]^3)$.  If this does not uniquely determine a 1-dimensional subspace then select one of the possibilities arbitrarily.
\item  $\overline{p_1}(\ii) , \underline{p_1}(\ii)\in [0,1]$ be given by $\underline{p_1}(\ii) = \bd \pi_\ii^1F$, noting that this exists since $\pi_\ii^1F$ is a (graph-directed) self-similar set, and $\overline{p_1}(\ii) = \qad \pi_\ii^1F$.
\end{enumerate}
\end{defn}

\begin{defn}[Principal two-dimensional projections]
Let $\ii \in \mathcal{I}^*$. We define:
\begin{enumerate}[(a)]
\item $\pi_\ii^2$ to be the projection  onto the 2-dimensional subspace perpendicular  to the shortest edges of the cuboid $S_\ii([0,1]^3)$.  If this does not uniquely determine a 2-dimensional subspace then select one of the possibilities arbitrarily.
\item   $\overline{p_2}(\ii) , \underline{p_2}(\ii)\in [0,2]$ be given by $\overline{p_2}(\ii)  = \ubd \pi_\ii^2F$ and $\underline{p_2}(\ii)  = \lbd \pi_\ii^2F$. 
\end{enumerate}
\end{defn}

Note that we do not know if $\overline{p_2}(\ii) = \underline{p_2}(\ii)$ although this would follow from a very special case of the well-known and difficult conjecture that the box dimension of any (graph-directed) self-affine set exists, irrespective of overlaps. We also do not know if $\overline{p_1}(\ii) = \underline{p_1}(\ii)$ although this would also follow from a conjecture that the box dimension and quasi-Assouad dimension  of any (graph-directed) self-similar set coincide, irrespective of overlaps.  This is known to be true in many special cases,  which are useful to us.  For example, it holds if the defining system satisfies the \emph{graph-directed weak separation property} (which includes the graph-directed open set condition) \cite{FraserOrponen} or if the  system has no \emph{super exponential concentration of cylinders}     \cite{fraseryu}.   In fact, a flawed proof of the full result  appeared in the original version of  \cite{Hare}.  Note that our use of quasi-Assouad dimension over Assouad dimension is important because self-similar sets may have distinct box and Assouad dimensions, see \cite{fraserrobinson}.

Also note that there are only three possible values for    each of $\overline{p_1}(\ii)$ $\underline{p_1}(\ii)$ $\overline{p_2}(\ii)$ and $\underline{p_2}(\ii)$, which are the  dimensions of the projection of $F$ onto the three principal 1-dimensional subspaces and the three principal 2-dimensional subspaces, respectively. 

Finally observe that $S_3$-sponges, $S_2$-partially ordered sponges and  ordered sponges all have  the property that $\overline{p_1}(\ii)$, $\underline{p_1}(\ii)$, $\overline{p_2}(\ii)$, $\underline{p_2}(\ii)$ are all independent of $\ii$ and therefore in these settings we can suppress the dependence on $\ii$ and write $\overline{p_1}$, $\underline{p_1}$, $\overline{p_2}$, $\underline{p_2}$. In the  $S_3$-sponges case, this follows by monotonicity of these dimensions since each co-ordinate projection contains a bi-Lipschitz image of every other  co-ordinate projection. For ordered sponges this follows from the co-ordinate ordering condition, which guarantees we will always project onto the same subspace.  In the $S_2$-partially ordered case it follows since the partial order means we are always projecting onto the same 2-dimensional subspace and one of two of the 1-dimensional subspaces.  Then the  $S_2$ property guarantees that the projections onto the two 1-dimensional subspaces have the same dimensions.  Also notice that ordered and  separated sponges have the property that $\underline{p_1}=\overline{p_1}$ and $\underline{p_2}=\overline{p_2}$, since the relevant 2-dimensional projection corresponds to a self-affine carpet that satisfies the ROSC and therefore the box dimension exists by \cite{fraser}, and the relevant  1-dimensional projection is a self-similar set satisfying the OSC and so has equal box and quasi-Assouad dimension.

\section{Modified singular value functions and pressure} \label{dimformula}

Our dimension results will be expressed in terms of a (modified) singular value function, as in \cite{fraser}. To this end, we now introduce the modified singular value function which is natural in our setting.

For $s\geq 0$ and $\textbf{\emph{i}} \in \mathcal{I}^*$, we define upper and lower \emph{modified singular value functions},  by
\begin{equation} \label{modsing}
\overline{\psi}^s\big(S_{\textbf{\emph{i}}}\big) = \alpha_1 (\textbf{\emph{i}})^{ \overline{p_1}(\ii)} \, \,  \alpha_2 (\textbf{\emph{i}})^ {\overline{p_2}(\ii)-\overline{p_1}(\ii)}\alpha_3(\ii)^{s-\overline{p_2}(\ii)},
\end{equation}
and
\begin{equation} \label{modsing}
\underline{\psi}^s\big(S_{\textbf{\emph{i}}}\big) = \alpha_1 (\textbf{\emph{i}})^{ \underline{p_1}(\ii)} \, \,  \alpha_2 (\textbf{\emph{i}})^ {\underline{p_2}(\ii)- \underline{p_1}(\ii)}\alpha_3(\ii)^{s-\underline{p_2}(\ii)},
\end{equation}
respectively. For $s\geq 0$ and $k \in \mathbb{N}$, we define 
\[
\overline{\Psi}_k^s= \sum_{\textbf{\emph{i}} \in \mathcal{I}^{k}} \overline{\psi}^s(S_{\textbf{\emph{i}}}) 
\]
and
\[
\underline{\Psi}_k^s= \sum_{\textbf{\emph{i}} \in \mathcal{I}^{k}} \underline{\psi}^s(S_{\textbf{\emph{i}}}) .
\]
We would now like to define upper and lower `pressure functions' as $\overline{P}(s)=\lim_{k \to \infty} (\overline{\Psi}_k^s)^{\frac{1}{k}}$ and $\underline{P}(s)=\lim_{k \to \infty} (\underline{\Psi}_k^s)^{\frac{1}{k}}$. Usually, this is possible as a result of the sub-multiplicativity or super-multiplicativity properties of the associated singular value function, see for instance \cite{falconer1, fraser}. However, in general our modified singular value functions do not have to be neither sub-multiplicative nor super-multiplicative for certain values of $s$, see Remark \ref{mult}. Instead we can guarantee existence of the pressure by passing to an appropriate subsystem where the modified singular value function is multiplicative. The following lemma guarantees the existence of a multiplicative subsystem that `carries the pressure', which will allow us to prove the existence of $\overline{P}(s)$ and $\underline{P}(s)$.

\begin{lma} \label{carries lemma}
Fix $s>0$. For all $N \in \N$, there exist $\overline{\Gamma}, \underline{\Gamma} \subset \mathcal{I}^{N+m}$ for some $0 \leq m \leq 3$ such that:
  \begin{eqnarray}
 \overline{\psi}^s(\ii\jj)=\overline{\psi}^s(\ii)\overline{\psi}^s(\jj) & &\forall \ii, \jj \in \overline{\Gamma} \nonumber\\
 \underline{\psi}^s(\ii\jj)=\underline{\psi}^s(\ii)\underline{\psi}^s(\jj) & &\forall \ii, \jj \in \underline{\Gamma} \label{multi}
\end{eqnarray}
and
 \begin{eqnarray}
\sum_{\ii \in \overline{\Gamma}} \overline{\psi}^s(\ii) &\geq& C \sum_{\ii \in \mathcal{I}^N} \overline{\psi}^s(\ii) \nonumber\\
\sum_{\ii \in \underline{\Gamma}} \underline{\psi}^s(\ii) &\geq& C \sum_{\ii \in \mathcal{I}^N} \underline{\psi}^s(\ii) \label{carries}
\end{eqnarray}
where $C$ depends only on $s$.
\end{lma}

\begin{proof}
We will prove the result for $\underline{\psi}^s$, but the proof for $\overline{\psi}^s$ is almost identical. We begin by showing  that for all $s>0$ and  $M \in \N$, there exists $c>0$ that depends only on $s$ and $M$  such that for all $\jj \in \mathcal{I}^*$ with  $|\jj| = M$ and all $\ii \in \mathcal{I}^*$,
\begin{equation}
\underline{\psi}^s(\ii\jj) \geq c\underline{\psi}^s(\ii). \label{small drop}
\end{equation}
Fix $s>0$, $M \in \N$, $\jj \in \mathcal{I}^M$ and $\ii \in \mathcal{I}^*$. If $\underline{p_1}(\ii\jj)\geq \underline{p_1}(\ii)$, then 
\[
\left(\frac{\alpha_1(\ii\jj)}{\alpha_2(\ii\jj)}\right)^{\underline{p_1}(\ii\jj)} \geq \left(\frac{\alpha_1(\ii\jj)}{\alpha_2(\ii\jj)}\right)^{\underline{p_1}(\ii)}
\]
 and if $\underline{p_2}(\ii\jj) \geq \underline{p_2}(\ii)$, then 
\[
\left(\frac{\alpha_2(\ii\jj)}{\alpha_3(\ii\jj)}\right)^{\underline{p_2}(\ii\jj)} \geq \left(\frac{\alpha_2(\ii\jj)}{\alpha_3(\ii\jj)}\right)^{\underline{p_2}(\ii)}.
\]
On the other hand, it is easy to see that there exist constants $c_1, c_2>1$ which depend on $M$ but are independent of $\ii$ and $\jj$ such that if $\underline{p_1}(\ii\jj) \neq \underline{p_1}(\ii)$ then $\frac{\alpha_1(\ii\jj)}{\alpha_2(\ii\jj)} \leq c_1$ and if $\underline{p_2}(\ii\jj) \neq \underline{p_2}(\ii)$ then $\frac{\alpha_2(\ii\jj)}{\alpha_3(\ii\jj)} \leq c_2$.

Therefore, if $\underline{p_1}(\ii\jj)<\underline{p_1}(\ii)$ then
$$\left(\frac{\alpha_1(\ii\jj)}{\alpha_2(\ii\jj)}\right)^{\underline{p_1}(\ii\jj)}=\left(\frac{\alpha_1(\ii\jj)}{\alpha_2(\ii\jj)}\right)^{\underline{p_1}(\ii\jj)-\underline{p_1}(\ii)}\left(\frac{\alpha_1(\ii\jj)}{\alpha_2(\ii\jj)}\right)^{\underline{p_1}(\ii)} \geq \frac{1}{c_1} \left(\frac{\alpha_1(\ii\jj)}{\alpha_2(\ii\jj)}\right)^{\underline{p_1}(\ii)}$$
and if $\underline{p_2}(\ii\jj)<\underline{p_2}(\ii)$ then
$$\left(\frac{\alpha_2(\ii\jj)}{\alpha_3(\ii\jj)}\right)^{\underline{p_2}(\ii\jj)}=\left(\frac{\alpha_2(\ii\jj)}{\alpha_3(\ii\jj)}\right)^{\underline{p_2}(\ii\jj)-\underline{p_2}(\ii)}\left(\frac{\alpha_2(\ii\jj)}{\alpha_3(\ii\jj)}\right)^{\underline{p_2}(\ii)} \geq \frac{1}{c_2^2} \left(\frac{\alpha_2(\ii\jj)}{\alpha_3(\ii\jj)}\right)^{\underline{p_2}(\ii)}.$$
Therefore, for any $\ii \in \mathcal{I}^*$ and $\jj \in \mathcal{I}^M$,
\begin{eqnarray*}
\underline{\psi}^s(\ii\jj)&=&\left(\frac{\alpha_1(\ii\jj)}{\alpha_2(\ii\jj)}\right)^{\underline{p_1}(\ii\jj)}\left(\frac{\alpha_2(\ii\jj)}{\alpha_3(\ii\jj)}\right)^{\underline{p_2}(\ii\jj)} \alpha_3(\ii\jj)^s \\
&\geq& \frac{1}{c_1c_2^2}\left(\frac{\alpha_1(\ii\jj)}{\alpha_2(\ii\jj)}\right)^{\underline{p_1}(\ii)}\left(\frac{\alpha_2(\ii\jj)}{\alpha_3(\ii\jj)}\right)^{\underline{p_2}(\ii)} \alpha_3(\ii\jj)^s \\
&\geq& \frac{\alpha_{min}^{Ms}}{c_1c_2^2} \alpha_1(\ii)^{\underline{p_1}(\ii)}\alpha_2(\ii)^{\underline{p_2}(\ii)-\underline{p_1}(\ii)}\alpha_3(\ii)^{s-\underline{p_2}(\ii)} \ = \ \frac{\alpha_{min}^{Ms}}{ c_1c_2^2} \, \underline{\psi}^s(\ii),
\end{eqnarray*}
completing the proof of the claim \eqref{small drop}. 

Now, fix $N \in \N$. Observe that there exists $\Gamma_1 \subset \mathcal{I}^N$ such that each $\ii \in \Gamma_1$ corresponds to the same permutation of the axes (of which there are at most 6 possibilities) and
$$\sum_{\ii \in \Gamma_1}\underline{\psi}^s(\ii) \geq \frac{1}{6} \sum_{\ii \in \mathcal{I}^N} \underline{\psi}^s(\ii).$$

Next, observe that there exists $0 \leq m \leq 3$ such that for all $\ii \in \Gamma_1$ there exists $\jj(\ii) \in \mathcal{I}^m$ such that $A_{\ii\jj(\ii)}$ is a diagonal matrix. Denote $\Gamma_2=\{\ii\jj(\ii): \ii \in \Gamma_1\} \subset \mathcal{I}^{N+m}$. In particular, by taking $c$ to be the constant from (\ref{small drop}), $\underline{\psi}^s(\ii\jj(\ii)) \geq c \underline{\psi}^s(\ii)$ and therefore
$$\sum_{\ii \in \Gamma_2}\underline{\psi}^s(\ii) \geq c \sum_{\ii \in \Gamma_1}\underline{\psi}^s(\ii) \geq \frac{c}{6} \sum_{\ii \in \mathcal{I}^N} \underline{\psi}^s(\ii).$$

Finally, notice that there exists a subset $\underline{\Gamma} \subset \Gamma_2$ such that all the matrices in $\{A_{\ii}: \ii \in \underline{\Gamma}\}$ have the same ordering (of which there are at most 6 possibilities) and 
\begin{eqnarray}\sum_{\ii \in \underline{\Gamma}}\underline{\psi}^s(\ii) \geq \frac{1}{6} \sum_{\ii \in \Gamma_2}\underline{\psi}^s(\ii) \geq \frac{c}{36} \sum_{\ii \in \mathcal{I}^N}\underline{\psi}^s(\ii). \label{gamma bound}
\end{eqnarray}
Since all of the matrices in $\{A_{\ii}: \ii \in \underline{\Gamma}\}$ are diagonal and have the same ordering, $\underline{\psi}^s(\ii_1\ii_2)=\underline{\psi}^s(\ii_1)\underline{\psi}^s(\ii_2)$ for $\ii_1$, $\ii_2 \in \underline{\Gamma}$, so by (\ref{gamma bound}) the proof is complete choosing $C= \frac{c}{36}$.
\end{proof}

\begin{prop}[Existence of pressure functions]
For all $s>0$, the \emph{upper} and \emph{lower} \emph{pressure functions} $\overline{P}(s)=\lim_{k \to \infty} (\overline{\Psi}_k^s)^{\frac{1}{k}}$ and $\underline{P}(s)=\lim_{k \to \infty} (\underline{\Psi}_k^s)^{\frac{1}{k}}$ exist.
\label{pressure}
\end{prop}

\begin{proof}
We will again only prove the proposition for $\underline{\psi}^s$, but the proof for $\overline{\psi}^s$ is almost identical. Fix $\epsilon>0$ and $N \in \N$ such that 
$$\left(\sum_{\ii \in \mathcal{I}^N} \underline{\psi}^s(\ii)\right)^{\frac{1}{N}} \geq \limsup_{n \to \infty} \left(\sum_{\ii \in \mathcal{I}^n}\underline{\psi}^s(\ii)\right)^{\frac{1}{n}}-\epsilon,$$
noting that $N$ can be chosen arbitarily large.

Let $C$ be the constant from Lemma \ref{carries lemma} and $\underline{\Gamma} \subset \mathcal{I}^{N+m}$ be chosen according to Lemma \ref{carries lemma} for some $0 \leq m \leq 3$. By (\ref{multi}) and (\ref{carries}), we deduce that for all $k \in \N$,
$$\sum_{\ii \in \underline{\Gamma}^k} \underline{\psi}^s(\ii)= \left(\sum_{\ii \in \underline{\Gamma}} \underline{\psi}^s(\ii)\right)^k \geq C^k \left(\sum_{\ii \in \mathcal{I}^N}\underline{\psi}^s(\ii) \right)^k$$
and therefore for all $k \in \N$,
\begin{eqnarray*}
\left(\sum_{\ii \in \mathcal{I}^{(N+m)k}} \underline{\psi}^s(\ii)\right)^{\frac{1}{(N+m)k}} &\geq& \left(\sum_{\ii \in \underline{\Gamma}^k}\underline{\psi}^s(\ii)\right)^{\frac{1}{(N+m)k}} \\
&\geq&C^{\frac{1}{N+m}} \left( \sum_{\ii \in \mathcal{I}^N}\underline{\psi}^s(\ii) \right)^{\frac{1}{N+m}} \\
&\geq&C^{\frac{1}{N+m}} \left(\limsup_{n \to \infty} \left(\sum_{\ii \in \mathcal{I}^n} \underline{\psi}^s(\ii)\right)^{\frac{1}{n}}-\epsilon\right)^{\frac{1}{1+\frac{m}{N}}}.
\end{eqnarray*}
In particular
$$\liminf_{k \to \infty} \left(\sum_{\ii \in \mathcal{I}^{(N+m)k}}\underline{\psi}^s(\ii)\right)^{\frac{1}{(N+m)k}} \geq C^{\frac{1}{N+m}} \left(\limsup_{n \to \infty} \left(\sum_{\ii \in \mathcal{I}^n} \underline{\psi}^s(\ii)\right)^{\frac{1}{n}}-\epsilon\right)^{\frac{1}{1+\frac{m}{N}}}.$$
We will now show that
$$\liminf_{k \to \infty} \left(\sum_{\ii \in \mathcal{I}^{(N+m)k}}\underline{\psi}^s(\ii)\right)^{\frac{1}{(N+m)k}} =\liminf_{n \to \infty}\left(\sum_{\ii \in \mathcal{I}^n}\underline{\psi}^s(\ii)\right)^{\frac{1}{n}}$$
which completes the proof. Observe that for any $n \in \N$ and $\ii \in \mathcal{I}^n$ there exists $0 \leq l < N+m$ such that $|\ii|_{n-l}|=(N+m)k'$ for some $k' \in \N$. By taking $c$ to be the constant from (\ref{small drop}) (which depends only on $N+m$), we have $\underline{\psi}^s(\ii) \geq c\underline{\psi}^s(\ii|_{n-l})$. In particular,
$$\sum_{\ii \in \mathcal{I}^n}\underline{\psi}^s(\ii) \geq  \sum_{\ii \in \mathcal{I}^n} c\underline{\psi}^s(\ii|_{n-l}) \geq c\sum_{\ii \in \mathcal{I}^{(N+m)k'}}\underline{\psi}^s(\ii).$$
It follows that 
\begin{eqnarray*}
\liminf_{n \to \infty} \left(\sum_{\ii \in \mathcal{I}^n}\underline{\psi}^s(\ii)\right)^{\frac{1}{n}} &\geq& \liminf_{k \to \infty} \left(c \sum_{\ii \in \mathcal{I}^{(N+m)k}} \underline{\psi}^s(\ii)\right)^{\frac{1}{(N+m)k}}\\
&=&\liminf_{k \to \infty} \left( \sum_{\ii \in \mathcal{I}^{(N+m)k}} \underline{\psi}^s(\ii)\right)^{\frac{1}{(N+m)k}}
\end{eqnarray*}
since $c$ is independent of $k$. 
Therefore 
$$\liminf_{n \to \infty}\left(\sum_{\ii \in \mathcal{I}^n}\underline{\psi}^s(\ii)\right)^{\frac{1}{n}} \geq C^{\frac{1}{N+m}} \left(\limsup_{n \to \infty} \left(\sum_{\ii \in \mathcal{I}^n}\underline{\psi}^s(\ii)\right)^{\frac{1}{n}}-\epsilon\right)^{\frac{1}{1+\frac{m}{N}}}$$
and since $\epsilon>0$ can be taken arbitrarily small and $N \in \N$ can be taken arbitrarily large it follows that $\limsup_{n \to \infty} \left(\sum_{\ii \in \mathcal{I}^n}\underline{\psi}^s(\ii)\right)^{\frac{1}{n}}= \liminf_{n \to \infty} \left(\sum_{\ii \in \mathcal{I}^n}\underline{\psi}^s(\ii)\right)^{\frac{1}{n}}$, which proves the result.
\end{proof}

It is easy to see that $\underline{P}(0) \in (1, \infty)$ and for all $s, \epsilon \geq 0$,
\begin{equation}
\alpha_{min}^{\epsilon} \underline{P}(s) \leq \underline{P}(s+\epsilon) \leq \alpha_{max}^{\epsilon}\underline{P}(s) \label{continuity} \end{equation}
and therefore $\underline{P}$ is continuous and strictly decreasing on $[0, \infty)$. By (\ref{continuity}), for sufficiently large $s$, $\underline{P}(s)<1$ and therefore there exists a unique value $\underline{s_0}>0$ such that $\underline{P}(\underline{s_0})=1$. By a similar argument we can define $\overline{s_0}>0$ to be the unique value such that $\overline{P}(\overline{s_0})=1$. 

\begin{rem}[Multiplicative properties of $\overline{\psi}^s$ and $\underline{\psi}^s$ for $S_3$-sponges] Let $F$ be an $S_3$-sponge. Recall that we say that a function $f$ on $\I^{\ast}$ is submultiplicative if for all $\ii, \jj \in \I^{\ast}$, $f(\ii\jj) \leq f(\ii)f(\jj)$, supermultiplicative if $f(\ii\jj) \geq f(\ii)f(\jj)$ and multiplicative if $f(\ii\jj)=f(\ii)f(\jj)$. Since $\alpha_1$ is a submultiplicative function, $\alpha_3$ is a supermultiplicative function, $\alpha_1\alpha_2\alpha_3$ is a multiplicative function and
\begin{eqnarray*}
\overline{\psi}^s(\ii)&=& (\alpha_1(\ii)\alpha_2(\ii)\alpha_3(\ii))^{\overline{p_2}-\overline{p_1}} \alpha_1(\ii)^{2\overline{p_1}-\overline{p_2}}\alpha_3(\ii)^{s+\overline{p_1}-2\overline{p_2}},
\end{eqnarray*}
it follows that $\overline{\psi}^s$ is submultiplicative when $0 \leq s \leq 2\overline{p_2}-\overline{p_1}$. In the special case that $\overline{p_2}=2\overline{p_1}$, $\overline{\psi}^s$ is supermultiplicative when $s > 2\overline{p_2}-\overline{p_1}$. However when $\overline{p_2} \neq 2\overline{p_1}$, $\overline{\psi}^s$ is not necessarily submultiplicative nor supermultiplicative for $s> 2\overline{p_2}-\overline{p_1}$. For a simple example of this fact, consider the IFS where the linear parts of the maps are the matrices
\[ A_1=\begin{pmatrix} \frac{1}{2}&0&0 \\ 0&0&-\frac{1}{5} \\ 0&\frac{1}{4}&0 \end{pmatrix} \quad \textrm{and} \quad A_2= \begin{pmatrix} 0&0&-\frac{1}{3} \\ 0&\frac{1}{7}&0 \\ \frac{1}{5}&0&0 \end{pmatrix}. \]
$A_1$ is a rotation about the $x$ axis by 90$^{\circ}$ and $A_2$ is a rotation about the $y$ axis by 90$^{\circ}$, therefore $G \cong S_3$. We have
\[
 A_1A_2= \begin{pmatrix} 0&0&-\frac{1}{6} \\ -\frac{1}{25}&0&0 \\ 0&\frac{1}{28}&0 \end{pmatrix} \quad \textrm{and} \quad  A_2A_1=\begin{pmatrix} 0&-\frac{1}{12}&0 \\ 0&0&-\frac{1}{35} \\ \frac{1}{10}&0&0 \end{pmatrix}.
\]
Since $\alpha_1(12)=\alpha_1(1)\alpha_1(2)=\frac{1}{6}$ and $\alpha_3(12)=\frac{1}{28}>\frac{1}{35}=\alpha_3(1)\alpha_3(2)$ it follows that $\overline{\psi}^s(12)> \overline{\psi}^s(1)\overline{\psi}^s(2)$ for any $s>2\overline{p_2}-\overline{p_1}$. On the other hand since $\alpha_3(21)=\alpha_3(2)\alpha_3(1)=\frac{1}{35}$ and $\alpha_1(21)=\frac{1}{10}<\frac{1}{6}=\alpha_1(2)\alpha_1(1)$ it follows that $\overline{\psi}^s(21) < \overline{\psi}^s(2)\overline{\psi}^s(1)$ for any $s> 2\overline{p_2}-\overline{p_1}$. That is, for any $s>2\overline{p_2}-\overline{p_1}$,
\begin{eqnarray}
\overline{\psi}^s(21)<\overline{\psi}^s(1)\overline{\psi}^s(2)<\overline{\psi}^s(12). \label{no mult}
\end{eqnarray}
The same analysis holds when considering the function $\underline{\psi}^s$ instead.
\label{mult}
\end{rem} 

\begin{rem}[Pressure functions for ordered sponges] \label{remarkorderedseparated}
When $F$ is an ordered sponge, by the co-ordinate ordering condition $\underline{\psi}^s$ and $\overline{\psi}^s$ are multiplicative and therefore $\underline{s_0}$ is the unique value satisfying
$$\underline{P}(\underline{s_0})=\sum_{i \in \mathcal{I}} \alpha_1(i)^{\underline{p_1}}\alpha_2(i)^{\underline{p_2}-\underline{p_1}} \alpha_3(i)^{\underline{s_0}-\underline{p_2}}=1$$
and
$\overline{s_0}$ is the unique value satisfying
$$\overline{P}(\overline{s_0})=\sum_{i \in \mathcal{I}} \alpha_1(i)^{\overline{p_1}}\alpha_2(i)^{\overline{p_2}-\overline{p_1}} \alpha_3(i)^{\overline{s_0}-\overline{p_2}}=1.$$
\end{rem}

\section{Main results} \label{results}

The following separation condition is the 3-dimensional analogue of the rectangular open set condition introduced in \cite{fengaffine}.

\begin{defn}
An IFS $\{S_i\}_{i\in \mathcal{I}}$ satisfies the cuboidal  open set condition (COSC) if there exists a non-empty open cuboid, $C = (a,b)\times(c,d)  \times(e,f)  \subset \mathbb{R}^3$, such that $\{S_i(C)\}_{i\in \mathcal{I}}$ are pairwise disjoint subsets of $C$.
\end{defn}

The following is our main result.

\begin{thm} \label{main1} 
If  $F$ is  a sponge, then
\[
\ubd F \leq \overline{s_0}.
\]
Moreover, if $F$ is an $S_3$-sponge, an $S_2$-partially ordered sponge, or an ordered sponge and, in addition, the COSC is satisfied, then
\[
\underline{s_0} \leq \lbd F \leq \ubd F \leq \overline{s_0}.
\]
In particular, if $\overline{p_1}=\underline{p_1}$  and $\overline{p_2}=\underline{p_2}$,  then the box dimension of $F$ exists and  $\bd F=\overline{s_0}=\underline{s_0}$.
\end{thm}

A special case of the above result is when $\underline{p_1}=\overline{p_1}=1$ and $\underline{p_2}=\overline{p_2}=2$, in which case the box dimension is equal to the affinity dimension.

\begin{cor}
Let $F$ be an $S_3$-sponge, an $S_2$-partially ordered sponge, or an ordered sponge, which also  satisfies the COSC. If $\underline{p_1}=\overline{p_1}=1$ and $\underline{p_2}=\overline{p_2}=2$ then $\bd F$ is the affinity dimension of $F$.
\end{cor}

It is natural to consider possible improvements in Theorem \ref{main1}.  It would be particularly useful to know that $\overline{p_1}=\underline{p_1}$  and $\overline{p_2}=\underline{p_2}$ in all cases.  Both of these statements are conjectured to be true, even in much more general settings than we require, but we are unable to prove either conjecture. However, there are many situations where either or both statements are true, but we decided not to state these situations as additional corollaries and leave their precise formulation to the interested reader.  For example, if the principal projections satisfy some separation condition (graph directed open set condition or weak separation condition in the $p_1$ case) then the statements hold.

As an example of the type of result we can get, we include one corollary where our result is very explicit.  This setting generalises the Lalley-Gatzouras family to $\mathbb{R}^3$. 

\begin{cor}
If $F$ is an ordered and  separated sponge that satisfies the COSC, then  $\bd F=\underline{s_0}=\overline{s_0}$ and moreover this common value is given by a closed form expression, see Remark \ref{remarkorderedseparated}. \label{lg}
\end{cor}

Clearly some sort of separation condition is required and so the appearance of the COSC in Theorem \ref{main1} to obtain lower bounds is innocuous. The most striking assumption we make in Theorem \ref{main1} is the requirement that the sponge be either $S_3$, $S_2$-partially ordered, or ordered.  In particular, this excludes the extension of the important Bara\'nski family to $\mathbb{R}^3$  where all of the matrices are diagonal, but the system is not ordered.  It was of great surprise to us that in fact our lower bound  does \emph{not} hold in this setting.

\begin{thm} \label{counter}
There exists a self-affine sponge generated by diagonal matrices and satisfying the COSC such that
$$\ubd F < \underline{s_0}=\overline{s_0}.$$
\end{thm}

\section{Some notation} \label{notation}

Throughout the rest of the paper we will write $A \lesssim B$ to mean there exists a uniform constant $c>0$ (depending only on parameters fixed by the hypotheses of the result being proved, such as the IFS) such that $A \leq c B$.  If  we wish to emphasise that $c$ depends on something else, not fixed by the hypotheses such as $\varepsilon$, then we write $A \lesssim_\varepsilon B$.  Similarly, we write $A \gtrsim B$ or $A \gtrsim_\varepsilon B$ to mean $B \lesssim A$ or $B \lesssim_\varepsilon A$, respectively.  Finally, we write $A \approx B$ if $A \lesssim B$ and $A \gtrsim B$.

When proving Theorem \ref{main1}, the basic idea is to cover cylinders $S_{\ii}([0,1]^3)$  separately  where the family of cylinders is chosen such that the smallest side of the cube $S_{\ii}([0,1]^3)$ is roughly the same scale as the cover.  This means that covering the cylinder is similar to covering an affine image of the projection of $F$ onto one of the principal 2-dimensional subspaces.  To this end, we introduce $\delta$-stoppings or `cut sets' as follows.  Given $\delta \in (0,1)$
\[
\mathcal{I}(\delta) = \{ \ii \in \mathcal{I}^* : \alpha_3(\ii) \leq \delta < \alpha_3(\ii^-) \}
\]
where, for $\ii = i_1 i_2 \cdots i_k$,  $\ii^- =  i_1 i_2 \cdots i_{k-1}$.

\section{Proof of upper bound in Theorem \ref{main1}} \label{ub}

In this section we prove the upper bound in Theorem \ref{main1} for which we require no assumptions on our sponge at all.  In particular, any type of sponge is permitted and no separation conditions are required.

\subsection{Covering lemma}

\begin{lma} \label{keycover1}
Fix $\eps \in (0,1)$.  For $\delta \in (0,1)$ and $\ii \in \mathcal{I}(\delta)$, we have
\[
N_\delta(S_\ii(F)) \lesssim_\eps \left( \frac{ \alpha_1(\ii)}{\alpha_2(\ii)} \right)^{\overline{p_1}(\ii)} \left( \frac{ \alpha_2(\ii)}{ \alpha_3(\ii)}\right)^{\overline{p_2}(\ii) }\alpha_3(\ii)^{-\eps}.
\]
\end{lma}

\begin{proof}
Since $\alpha_3(\ii) \leq \delta$, we have
\[
N_\delta(S_\ii(F)) \lesssim N_\delta(\pi_\ii^2  S_\ii(F) ).
\]
In order to obtain a $\delta$-cover of  $\pi_\ii^2  S_\ii(F)$ first take a  cover of $S_\ii^{-1}\pi_\ii^2  S_\ii(F)$ by
\[
\lesssim_\eps \left(\frac{\alpha_2(\ii)}{\delta}\right)^{\overline{p_2}(\ii) +\eps}
\]
 squares of sidelength $\delta/ \alpha_2(\ii)$ and consider the image of this cover under $S_\ii$ which is a cover of $\pi_\ii^2  S_\ii(F)$ by rectangles with sidelengths $\delta \alpha_1(\ii)/ \alpha_2(\ii) $    and $\delta$.  We want to cover each of these rectangles efficiently by squares of sidelength $\delta$ and the key observation is that this is equivalent to covering some interval of length $\delta \alpha_1(\ii)/ \alpha_2(\ii)$ intersected with $\pi_\ii^1  S_\ii(F)$ with intervals of length $\delta$.  Moreover, this is equivalent to covering some interval of length $\delta  / \alpha_2(\ii)$ intersected with $S_\ii^{-1} \pi_\ii^1  S_\ii(F)$ with intervals of length $\delta/\alpha_1(\ii)$.  The set $F' = S_\ii^{-1} \pi_\ii^1  S_\ii(F)$ is simply the projection of $F$ onto  one of the principal 1-dimensional subspaces.  Moreover, $h_{F'}(\eps) \leq \qad F' = \overline{p_1}(\ii) $.  Therefore, if 
\begin{equation} \label{kappaa}
\frac{\delta}{\alpha_1(\ii)}  \leq \left( \frac{\delta}{ \alpha_2(\ii)} \right)^{1+\eps}
\end{equation}
then we may cover each rectangle above by
\begin{equation} \label{good}
\lesssim_\eps  \left( \frac{ \alpha_1(\ii)}{ \alpha_2(\ii)} \right)^{h_{F'}(\eps)+\eps} \leq   \left( \frac{ \alpha_1(\ii)}{ \alpha_2(\ii)} \right)^{ \overline{p_1}(\ii)   +\eps} 
\end{equation}
squares of  sidelength $\delta$. Chaining the above estimates we get
\begin{eqnarray*}
N_\delta(S_\ii(F))   &\lesssim_\eps& \left(\frac{\alpha_2(\ii)}{\delta}\right)^{\overline{p_2}(\ii) +\eps}\left( \frac{ \alpha_1(\ii)}{ \alpha_2(\ii)} \right)^{ \overline{p_1}(\ii)   +\eps} \\ \\
& \lesssim&  \left( \frac{ \alpha_1(\ii)}{\alpha_2(\ii)} \right)^{\overline{p_1}(\ii)} \left( \frac{ \alpha_2(\ii)}{ \alpha_3(\ii)}\right)^{\overline{p_2}(\ii) } \alpha_3(\ii)^{-\eps}
\end{eqnarray*}
as required.  Now suppose that \eqref{kappaa} is not satisfied, in which case we must replace \eqref{good} with  the trivial estimate 
\[
\lesssim_\eps  \left( \frac{ \alpha_1(\ii)}{ \alpha_2(\ii)} \right)
\]
which gives
\begin{eqnarray}
N_\delta(S_\ii(F))   &\lesssim_\eps& \left(\frac{\alpha_2(\ii)}{\delta}\right)^{\overline{p_2}(\ii) +\eps}\left( \frac{ \alpha_1(\ii)}{ \alpha_2(\ii)} \right)  \nonumber \\  \nonumber\\
&\lesssim&  \left(\frac{\alpha_2(\ii)}{\alpha_3(\ii)}\right)^{\overline{p_2}(\ii)  }\left( \frac{ \alpha_1(\ii)}{ \alpha_2(\ii)} \right)^{\overline{p_1}(\ii)  }   \alpha_3(\ii)^{-\eps}\left( \frac{ \alpha_1(\ii)}{ \alpha_2(\ii)} \right)^{1-\overline{p_1}(\ii)  }. \label{upper1}
\end{eqnarray}
However, since \eqref{kappaa} fails,  
\[
\frac{\alpha_3(\ii)}{\alpha_1(\ii)}  \gtrsim \left( \frac{\alpha_3(\ii)}{ \alpha_2(\ii)} \right)^{1+\eps}
\]
and, in particular, 
\[
\frac{ \alpha_1(\ii)}{ \alpha_2(\ii)} \lesssim  \left(\frac{ \alpha_2(\ii)}{ \alpha_3(\ii)}\right)^\eps   \leq   \alpha_3(\ii)^{-\eps}.
\]
Combining this with \eqref{upper1} we get
\begin{eqnarray*}
N_\delta(S_\ii(F))   &\lesssim & \left(\frac{\alpha_2(\ii)}{\alpha_3(\ii)}\right)^{\overline{p_2}(\ii)}\left( \frac{ \alpha_1(\ii)}{ \alpha_2(\ii)} \right)^{\overline{p_1}(\ii)} \alpha_3(\ii)^{-2\eps}
\end{eqnarray*}
This is sufficient to prove the lemma, replacing $2\eps$ with $\eps$.
\end{proof}

\subsection{Proof of upper bound}

Let $\varepsilon >0$, $\delta \in (0,1]$ and $s>\overline{s_0}$. 
\begin{eqnarray*}
 \delta^{s+\varepsilon} N_{\delta} (F)
&\leq& \delta^{s+\varepsilon} \sum_{\textbf{\emph{i}} \in \mathcal{I}(\delta)} N_{\delta} \big(S_{\ii}(F)\big) \\
&\lesssim_\eps & \delta^{s+\epsilon} \sum_{\textbf{\emph{i}} \in \mathcal{I}(\delta)} \left( \frac{ \alpha_1(\ii)}{\alpha_2(\ii)} \right)^{\overline{p_1}(\ii)} \left( \frac{ \alpha_2(\ii)}{ \alpha_3(\ii)}\right)^{\overline{p_2}(\ii) }\alpha_3(\ii)^{-\eps} \qquad \text{(by Lemma \ref{keycover1})} \\
&\lesssim &  \sum_{\textbf{\emph{i}} \in \mathcal{I}(\delta)} \left( \frac{ \alpha_1(\ii)}{\alpha_2(\ii)} \right)^{\overline{p_1}(\ii)} \left( \frac{ \alpha_2(\ii)}{ \alpha_3(\ii)}\right)^{\overline{p_2}(\ii) }\alpha_3(\ii)^{-\eps} \alpha_3(\ii)^{s+\eps}\\
&=& \sum_{\textbf{\emph{i}} \in \mathcal{I}(\delta)} \overline{\psi}^s(\ii) \\
&\leq& \sum_{k=1}^{\infty} \overline{\Psi}_k^s.
\end{eqnarray*}
Since $s>\overline{s_0}$, $\overline{P}(s)= \lim_{k \to \infty} (\Psi_k^s)^{\frac{1}{k}}<1$, therefore $\sum_{k=1}^{\infty} \overline{\Psi}_k^s< \infty$. It follows that $\overline{\dim}_\text{B} F \leq s+\varepsilon$. Since $\varepsilon>0$ and $s>\overline{s_0}$ were arbitrary, we have the desired upper bound. \qed

\section{Proof of lower bound in Theorem \ref{main1}} \label{lbsection}
 
In this section we prove the lower bound from Theorem \ref{main1}.  For this we assume the COSC is satisfied and that the sponge is either $S_3$, $S_2$-partially ordered, or ordered.  The following lemma gives us control on the size of $\sum_{\ii \in \mathcal{I}(\delta)} \underline{\psi}^s(\ii)$ whenever $s<\underline{s_0}$ and holds without these additional assumptions. 

\begin{lma} \label{lb lma}
Let $F$ be any sponge and fix $s< \underline{s_0}$. For all $\delta>0$,
$$ \sum_{\ii \in \mathcal{I}(\delta)} \underline{\psi}^s(\ii) \gtrsim_s 1.$$
\end{lma}

\begin{proof}
Fix $s<\underline{s_0}$. Let $C$ be the constant from Lemma \ref{carries lemma}. Since $\underline{P}(s)>1$ there exists $N$ sufficiently large such that
\begin{equation} C \sum_{\ii \in \mathcal{I}^N} \underline{\psi}^s(\ii) \geq 1 . \label{carrieseq} \end{equation}
We fix this value of $N$. By Lemma \ref{carries lemma} there exists $\underline{\Gamma} \subset \mathcal{I}^{N+m}$ with $0 \leq m \leq 3$ such that $\underline{\psi}^s(\ii\jj)=\underline{\psi}^s(\ii)\underline{\psi}^s(\jj)$ for all $\ii, \jj \in \underline{\Gamma}$ and 
$$\sum_{\ii \in \underline{\Gamma}}\underline{\psi}^s(\ii) \geq C \sum_{\ii \in \mathcal{I}^N} \underline{\psi}^s(\ii).$$
By (\ref{carrieseq}),
\begin{equation}
\sum_{\ii \in \underline{\Gamma}} \underline{\psi}^s(\ii) \geq 1.
\label{keybound}
\end{equation}
Now fix $\delta>0$ sufficiently small that $\delta<\alpha_{max}^N$. In particular this means that for all $\ii \in \mathcal{I}(\delta)$, $|\ii| >N$. Write $k(\ii)= |\ii|-(|\ii| \mod N)$. Note that by (\ref{small drop}), $\underline{\psi}^s(\ii) \geq c \underline{\psi^s}(\ii|_{k(\ii)})$ where $c$ is a constant which depends only on $s$ and $N$, in particular it does not depend on $\delta$. Let $\mathcal{I}_{\underline{\Gamma}}(\delta) \subset \mathcal{I}(\delta)$ be the collection of $\ii \in \mathcal{I}(\delta)$ for which $\ii|_{k(\ii)} \in \underline{\Gamma}^l$ for some $l \in \N$. Let 
$$\mathcal{I}^-(\delta)= \{ \ii|_{k(\ii)}: \ii \in \mathcal{I}_{\underline{\Gamma}}(\delta)\}.$$
It follows that
$$\sum_{\ii \in \mathcal{I}(\delta)} \underline{\psi}^s(\ii) \geq \sum_{\ii \in \mathcal{I}_{\underline{\Gamma}}(\delta)} \underline{\psi}^s(\ii) \geq c \sum_{\ii \in \mathcal{I}^-(\delta)} \underline{\psi}^s(\ii).$$
Fix $k= \min\{k(\ii): \ii \in \mathcal{I}^-(\delta)\}$. By (\ref{keybound}) and the fact that $\underline{\psi}^s$ is multiplicative on $\underline{\Gamma}$, 
$$\sum_{\ii \in \mathcal{I}^-(\delta)} \underline{\psi}^s(\ii) \geq \sum_{\ii \in \underline{\Gamma}^k} \underline{\psi}^s(\ii)= \left(\sum_{\ii \in \underline{\Gamma}} \underline{\psi}^s(\ii) \right)^k \geq 1.$$
In particular, $\sum_{\ii \in \mathcal{I}(\delta)} \underline{\psi}^s(\ii) \geq c$, completing the proof.
\end{proof}

\subsection{Covering lemma}

In the following lemma we obtain a lower bound on the covering number $N_{\delta}(S_{\ii}F)$ for $S_3$-sponges, $S_2$-partially ordered sponges and ordered sponges. 

\begin{lma} \label{keycover2}
Suppose $F$ is an $S_3$-sponge, an $S_2$-partially ordered sponge or an ordered sponge and fix $\eps \in (0,1)$.  For $\delta \in (0,1)$ and $\ii \in \mathcal{I}(\delta)$, we have
\[
N_\delta(S_\ii(F)) \gtrsim_\eps \left( \frac{ \alpha_1(\ii)}{\alpha_2(\ii)} \right)^{\underline{p_1}} \left( \frac{ \alpha_2(\ii)}{ \alpha_3(\ii)}\right)^{\underline{p_2}}\alpha_3(\ii)^{\eps}.
\]
\end{lma}

\begin{proof}
Since $\alpha_3(\ii) \leq \delta$, we have
\[
N_\delta(S_\ii(F)) \gtrsim N_\delta(\pi_\ii^2  S_\ii(F) ).
\]
Since we are trying to bound this quantity from below it is more convenient to argue in terms of packings rather than covers.  We say a collection of squares oriented with the coordinate axes is an  $r$-packing of $E$ if the squares are centred in $E$, have sidelength $r$ and are pairwise disjoint.    In particular, if $X$ is such a packing, then $N_r(E) \approx \# X$.

 In order to obtain a $\delta$-packing of  $\pi_\ii^2  S_\ii(F)$ first take a  packing of $S_\ii^{-1}\pi_\ii^2  S_\ii(F)$ by
\[
\gtrsim_\eps \left(\frac{\alpha_2(\ii)}{\delta}\right)^{\underline{p_2}-\eps}
\]
 squares $\mathcal{Q}$ of sidelength $\delta/ \alpha_2(\ii)$ and consider the image of this packing under $S_\ii$ which is a packing of $\pi_\ii^2  S_\ii(F)$ by rectangles with sidelengths $\delta \alpha_1(\ii)/ \alpha_2(\ii) $    and $\delta$.   Given $Q \in \mathcal{Q}$ fix $\ii(Q) \in \mathcal{I}^*$ such that the rectangle $Q' :=S_\ii^{-1}\pi_\ii^2 S_\ii S_{\ii(Q)}([0,1]^3)$ intersects $Q$ and has longest side $\approx   \delta/ \alpha_2(\ii)$.
 
 \vspace{2mm}
 
\noindent \emph{Case 1 (Ordered sponge):} By the co-ordinate ordering condition, the longest sides of $S_{\ii}(Q)$ and $S_{\ii}(Q')$ are parallel. Therefore, we can find a $\delta$-packing inside $S_\ii(2Q)$ by
 \[
 \gtrsim_\eps \left(\frac{\delta \alpha_1(\ii)/ \alpha_2(\ii)}{\delta}\right)^{\underline{p_1} -\eps} = \left(\frac{ \alpha_1(\ii)}{\alpha_2(\ii)}\right)^{\underline{p_1} -\eps}
 \]
 many squares where $2Q$ is the square centred at the same point as $Q$ but with double the sidelength. Hence we can find a $\delta$-packing of $S_\ii(F)$ by
\[
\gtrsim_\eps \left(\frac{\alpha_2(\ii)}{\delta}\right)^{\underline{p_2}-\eps} \left(\frac{ \alpha_1(\ii)}{\alpha_2(\ii)}\right)^{\underline{p_1} -\eps} \gtrsim_\eps \left( \frac{ \alpha_1(\ii)}{\alpha_2(\ii)} \right)^{\underline{p_1}} \left( \frac{ \alpha_2(\ii)}{ \alpha_3(\ii)}\right)^{\underline{p_2}}\alpha_3(\ii)^{\eps}
\]
cubes and the proof is complete.

\vspace{2mm} 

\noindent \emph{Case 2 ($S_3$-sponge):}  There is now a ``good situation'' and a ``bad situation''.  In the good situation, the rectangles $S_\ii(Q')$ and $S_\ii(Q)$ are ``aligned'' in the same way, meaning that their longest sides are parallel, and in the bad situation their longest sides are orthogonal.  In the good situation, as before, we can find a $\delta$-packing inside $S_\ii(2Q)$ by
 \[
 \gtrsim_\eps \left(\frac{\delta \alpha_1(\ii)/ \alpha_2(\ii)}{\delta}\right)^{\underline{p_1} -\eps} = \left(\frac{ \alpha_1(\ii)}{\alpha_2(\ii)}\right)^{\underline{p_1} -\eps}
 \]
 many squares where $2Q$ is again the square centred at the same point as $Q$ but with double the sidelength. Therefore, if at least 1/2 of the squares in $\mathcal{Q}$ are in the good situation, then we can find a $\delta$-packing of $S_\ii(F)$ by
\[
\gtrsim_\eps \left(\frac{\alpha_2(\ii)}{\delta}\right)^{\underline{p_2}-\eps} \left(\frac{ \alpha_1(\ii)}{\alpha_2(\ii)}\right)^{\underline{p_1} -\eps} \gtrsim_\eps \left( \frac{ \alpha_1(\ii)}{\alpha_2(\ii)} \right)^{\underline{p_1}} \left( \frac{ \alpha_2(\ii)}{ \alpha_3(\ii)}\right)^{\underline{p_2}}\alpha_3(\ii)^{\eps}
\]
cubes and the proof would be complete.  Suppose that it is \emph{not} true that  at least 1/2 of the squares in $\mathcal{Q}$ are in the good situation, in which case at least half of the cubes are in the bad situation.  Since 
$G \cong S_3$, we can find $\jj \in \mathcal{I}^n$ for some $n \lesssim 1$ such that $A_\jj$ preserves the axis orthogonal to $\pi_\ii^2 \mathbb{R}^3$ and swaps the other two axes.  Consider the maps $S_\jj S_{\ii(Q)}$ for all $\ii(Q)$ corresponding to $Q$ in the bad situation.  By placing a centred square with sidelength $\approx  \delta/ \alpha_2(\ii)$ inside each of the sets $S_\ii^{-1}\pi_\ii^2 S_\ii S_\jj S_{\ii(Q)}([0,1]^3)$ we can find a packing of $S_\ii^{-1}\pi_\ii^2  S_\ii(F)$ by
\[
\gtrsim_\eps \left(\frac{\alpha_2(\ii)}{\delta}\right)^{\underline{p_2}-\eps}
\]
 squares $\mathcal{Q}$ of sidelength $\approx \delta/ \alpha_2(\ii)$ which are all now in the good situation by virtue of the map $S_\jj$ introducing a rotation.  This is enough to complete the proof.
\vspace{2mm} 

\noindent \emph{Case 3 ($S_2$-partially ordered sponge):} This is similar to the $S_3$ case, noting that the partial order guarantees that when we wish to  find $\jj \in \mathcal{I}^n$ for some $n \lesssim 1$ such that $A_\jj$ preserves the axis orthogonal to $\pi_\ii^2 \mathbb{R}^3$ and swaps the other two axes, we only need to consider $A_\jj$ which swap the first two coordinate axes and fixes the third.  The existence of such a $\jj$  is guaranteed by the $S_2$ assumption.
\end{proof}

\subsection{Proof of lower bound}

We are now ready to prove the lower bound in Theorem \ref{main1}. Let $0<s<\underline{s_0}$, $\varepsilon \in (0, s)$. Let $U$ be any closed cube of sidelength $\delta$.   Also, let $R$ be the open cuboid used in the COSC and let $r_-$ denote the length of the shortest side of $R$.  Finally, let
\[
M = \min\big\{n \in \mathbb{N} : n \geq (\alpha_{\min}r_-)^{-1}+2\big\}.
\]
Since $\{S_\textbf{\emph{i}}(R)\}_{\textbf{\emph{i}}\in \mathcal{I}(\delta)}$ is a collection of pairwise disjoint open rectangles each with shortest side having length at least $\alpha_{\min} \delta r_-$, it is clear that $U$ can intersect no more than $M^3$ of the sets $\{S_{\ii} F\}_{\textbf{\emph{i}}\in \mathcal{I}(\delta)}$.  It follows that, using the $\delta$-mesh definition of $N_\delta$, we have
\[
\sum_{\textbf{\emph{i}}\in \mathcal{I}(\delta)} N_\delta \big(F_{\textbf{\emph{i}}}\big)  \leq M^3 \, N_\delta (F) \approx N_\delta (F).
\]
This yields
\begin{eqnarray*}
\delta^{s-\varepsilon} N_\delta (F)  &\gtrsim&   \delta^{s-\varepsilon}\sum_{\textbf{\emph{i}}\in \mathcal{I}(\delta)} N_\delta \big(S_{\ii}(F)\big)\\ \\
&\gtrsim_{\epsilon} & \delta^{s-\varepsilon} \sum_{\textbf{\emph{i}}\in \mathcal{I}(\delta)} \left( \frac{ \alpha_1(\ii)}{\alpha_2(\ii)} \right)^{\underline{p_1}} \left( \frac{ \alpha_2(\ii)}{ \alpha_3(\ii)}\right)^{\underline{p_2} }\alpha_3(\ii)^{\eps} \qquad \text{(by Lemma \ref{keycover2})} \\
&\geq & \sum_{\textbf{\emph{i}}\in \mathcal{I}(\delta)} \left( \frac{ \alpha_1(\ii)}{\alpha_2(\ii)} \right)^{\underline{p_1}} \left( \frac{ \alpha_2(\ii)}{ \alpha_3(\ii)}\right)^{\underline{p_2} }\alpha_3(\ii)^{\eps}\alpha_3(\ii)^{s-\eps} \\
&=&  \sum_{\textbf{\emph{i}}\in \mathcal{I}(\delta)} \underline{\psi}^s(\ii)\gtrsim_s 1
\end{eqnarray*}
where the final line follows by Lemma \ref{lb lma}. Since $\eps \in (0,s)$ and $s< \underline{s_0}$ were arbitrary, $\underline{\dim}_\text{B} F \geq \underline{s_0}$.

\section{Example of dimension drop} \label{dimdrop}

In this section we prove Theorem \ref{counter}; in particular we construct a self-affine sponge which is generated by diagonal matrices and which satisfies the COSC, but whose upper box dimension is \emph{strictly smaller} than the root of the corresponding pressure function. 

Recall that in order to obtain a lower bound on the covering number $N_{\delta}(S_{\ii}F)$ in Lemma \ref{keycover2} one needed to guarantee that a positive proportion of packing boxes in the projection $\pi^2_{\ii} F$ intersected a cylinder with longest side comparable to the box side length and whose ordering agreed with the ordering of $\pi^2_{\ii}S_{\ii}F$.

For ordered sponges this followed directly from the co-ordinate ordering condition, which guarantees that \emph{any} cylinder in $\pi^2_{\ii}F$ agrees with the ordering of $\pi^2_{\ii}S_{\ii}F$. On the other hand, for $S_3$-sponges it followed from the existence of rotations which fixed the axis orthogonal to $\pi^2_{\ii}S_{\ii}F$ and swapped the other two axes so that at least half of packing boxes in the projection $\pi^2_{\ii} F$ intersected a cylinder with the desired properties.

However, for general sponges there is no way to guarantee that a positive proportion of packing boxes in the projection $\pi^2_{\ii} F$ intersect a cylinder with the desired properties and this can lead to a drop in the lower bound in Lemma \ref{keycover2}. In this section we construct a self-affine sponge $F$ such that `typical' cylinders in the projection $\pi^2_{\ii} F$ do not have an ordering that agrees with `typical' cylinders in $F$, which leads to a dimension drop.  

\subsection{Set-up} 

Let $0<\frac{1}{N}<c<b<a<d=1-b<1$ with $a+c<1$ and let $F$ be the attractor of the IFS $\{S_1, \ldots, S_{N+1}\}$ where:
\begin{equation*}
S_i\begin{pmatrix}x\\y\\ z\end{pmatrix}=
\left\{ \begin{array}{cc}

\begin{pmatrix}a&0&0\\0&b&0\\0&0&\frac{1}{N}\end{pmatrix}\begin{pmatrix}x\\y\\z\end{pmatrix}+\begin{pmatrix}0\\0\\\frac{i-1}{N}\end{pmatrix} & i=1, \ldots, N  \\
\begin{pmatrix}c&0&0\\0&d&0\\0&0&\frac{1}{N}\end{pmatrix}\begin{pmatrix}x\\y\\z\end{pmatrix}+\begin{pmatrix}1-c\\b\\0\end{pmatrix}& i=N+1,\end{array} \right.
\end{equation*}
see Figure \ref{jumpyfig}. We will show that for certain values of $a,b,c,d$ and $N$, $F$ is a self-affine sponge that satisfies Theorem \ref{counter}.

Let $A_{\ii}$ denote the linear part of the map $S_{\ii}$. Notice that for all $\ii \in \mathcal{I}^{\ast}$, $\pi_{\ii}^2$ corresponds to the projection to the Bara\'nski carpet $F'$ which is the attractor of the IFS $\{T_1, T_2\}$ with
\begin{eqnarray}
T_1\begin{pmatrix}x\\y\end{pmatrix}&=&\begin{pmatrix}a&0\\0&b\end{pmatrix}\begin{pmatrix}x\\y\end{pmatrix} \nonumber\\
T_2\begin{pmatrix}x\\y\end{pmatrix}&=&\begin{pmatrix}c&0\\0&d\end{pmatrix}\begin{pmatrix}x\\y\end{pmatrix}+\begin{pmatrix}1-c\\b\end{pmatrix} \label{baranski}.
\end{eqnarray}
 It follows from \cite[Theorem B]{baranski} that $\bd F'=1$ and therefore $\underline{p_2}(\ii)=\overline{p_2}(\ii)=1$ for all $\ii \in \mathcal{I}^{\ast}$. Also, $\underline{p_1}(\ii)=\overline{p_1}(\ii)$ and this will always either be equal to 1 or $t$, where $0<t<1$ satisfies $a^t+c^t=1$.

\begin{figure}[H]
  \begin{center}\includegraphics[width=\textwidth]{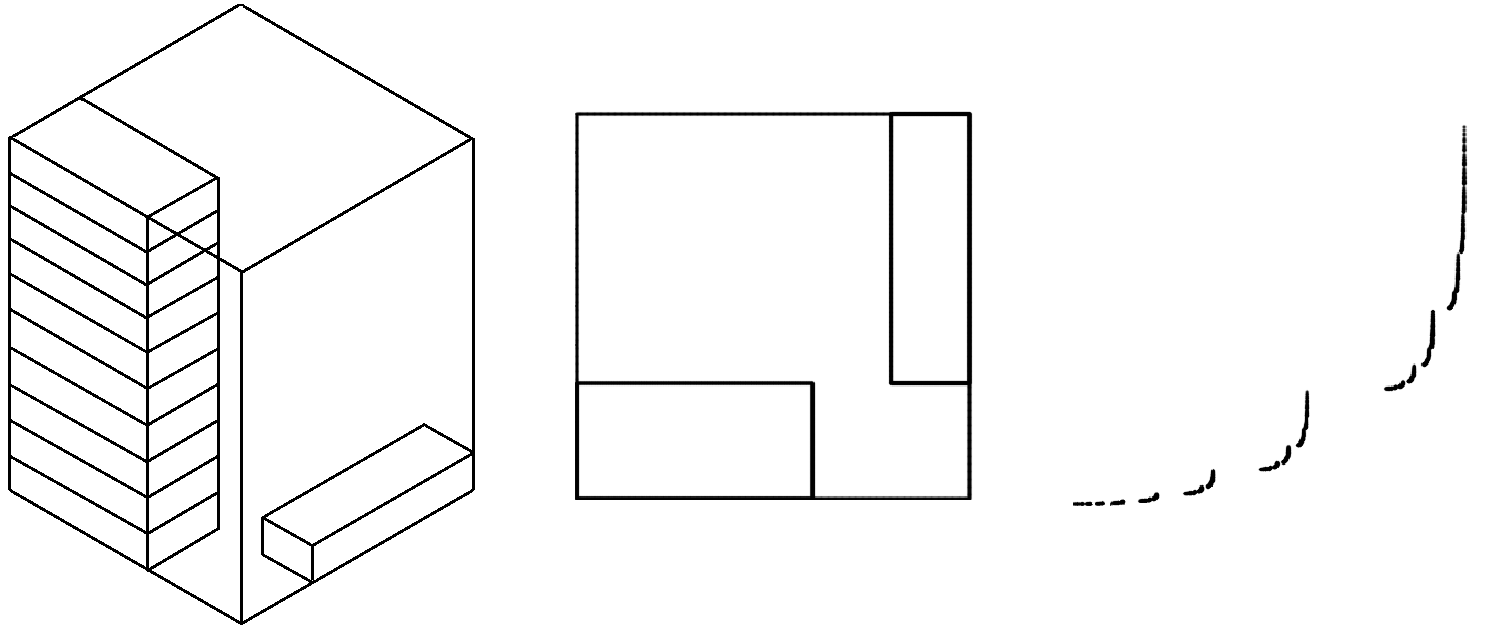}\end{center}
  \caption{Left: the images of the unit cube under the maps in the IFS. Centre: the image of the unit square under the maps generating the projection onto the first two co-ordinates. Right: the projection $F'$ of the attractor.  Here, $N=10$, $a=3/5$, $b=3/10$, $c=1/5$, and $d=7/10$.}
  \label{jumpyfig}
\end{figure}

Since $\underline{p_2}(\ii)=\overline{p_2}(\ii)=1$  and  $\underline{p_1}(\ii)=\overline{p_1}(\ii)$ for all $\ii \in \mathcal{I}^{\ast}$, we can write $\psi^s(\ii)= \alpha_1(\ii)^{p_1(\ii)}\alpha_2(\ii)^{1-p_1(\ii)}\alpha_3(\ii)^{s-1}$, where $p_1(\ii)=\underline{p_1}(\ii)=\overline{p_1}(\ii)$,  $\Psi_k^s= \sum_{\ii \in \mathcal{I}^k} \psi^s(\ii)$ and $P(s)= \lim_{k \to \infty} (\Psi_k^s)^{\frac{1}{k}}$. We can also write $s_0=\underline{s_0}=\overline{s_0}$ to be the unique value that satisfies $P(s_0)=1$. Moreover,
\begin{eqnarray*}
P(s) &=& \lim_{k \to \infty} \left( \sum_{\ii \in \mathcal{I}^k} \alpha_1(\ii)^{p_1(\ii)}\alpha_2(\ii)^{1-p_1(\ii)}\frac{1}{N^{k(s-1)}}\right)^{\frac{1}{k}} \\
&=&N^{1-s} \lim_{k \to \infty}  \left( \sum_{\ii \in \mathcal{I}^k} \alpha_1(\ii)^{p_1(\ii)}\alpha_2(\ii)^{1-p_1(\ii)}\right)^{\frac{1}{k}} \\
&=& N^{2-s}a^tb^{1-t}+c^td^{1-t}N^{1-s}
\end{eqnarray*}
for $N$ sufficiently large. To justify the final equality, observe that for all $\ii \in \mathcal{I}^k$
\[
\alpha_1(\ii)^{p_1(\ii)}\alpha_2(\ii)^{1-p_1(\ii)} = \max\{\beta_1(\ii)^{t}\beta_2(\ii)^{1-t} , \beta_2(\ii)\}
\]
where $\beta_1(\ii)$ is the length of the base of the rectangle $T_{\ii}([0,1]^2)$ and $\beta_2(\ii)$ is the height of the rectangle $T_{\ii}([0,1]^2)$. Therefore 
\begin{eqnarray*}
&\,& \hspace{-20mm}\max\left\{   \left(\sum_{\ii \in \mathcal{I}^k} \beta_1(\ii)^{t}\beta_2(\ii)^{1-t} \right)^{\frac{1}{k}}  , \  \left(\sum_{\ii \in \mathcal{I}^k} \beta_2(\ii)     \right)^{\frac{1}{k}}  \right\} \\ \\
 &\leq &    \left(\sum_{\ii \in \mathcal{I}^k} \alpha_1(\ii)^{p_1(\ii)}\alpha_2(\ii)^{1-p_1(\ii)}  \right)^{\frac{1}{k}} \ \leq \    \left(\sum_{\ii \in \mathcal{I}^k} \beta_1(\ii)^{t}\beta_2(\ii)^{1-t}   +  \sum_{\ii \in \mathcal{I}^k} \beta_2(\ii)  \right)^{\frac{1}{k}}  \\ \\
& \leq& 2^{\frac{1}{k}}  \max\left\{   \left(\sum_{\ii \in \mathcal{I}^k} \beta_1(\ii)^{t}\beta_2(\ii)^{1-t} \right)^{\frac{1}{k}}  , \  \left(\sum_{\ii \in \mathcal{I}^k} \beta_2(\ii)     \right)^{\frac{1}{k}}  \right\} .
\end{eqnarray*}
 Unlike the singular values,  $\beta_1(\ii)$  and $\beta_2(\ii)$ have the advantage of being multiplicative since the matrices defining $T_1, T_2$ are diagonal and therefore
\begin{eqnarray*}
\max\left\{   \left(\sum_{\ii \in \mathcal{I}^k} \beta_1(\ii)^{t}\beta_2(\ii)^{1-t} \right)^{\frac{1}{k}}  , \  \left(\sum_{\ii \in \mathcal{I}^k} \beta_2(\ii)     \right)^{\frac{1}{k}}  \right\} &=& \max\left\{ N a^tb^{1-t}+c^td^{1-t}  , \  N b+d  \right\} \\ \\
&=&N a^tb^{1-t}+c^td^{1-t}  
\end{eqnarray*}
for sufficiently large $N$, proving the claimed equality. 

It is easy to see that as $N \to \infty$, the root $s_0 \to 2$.  We also assume that the defining parameters satisfy the following conditions, noting that it is possible to satisfy them all simultaneously:
\begin{enumerate}[(i)]
\item $c$ is taken sufficiently small so that 
$$a^t\log a+c^t\log c< a^t \log b+c^t \log d,$$
\item $\eta'>0$ is taken sufficiently small such that
\begin{equation} \label{defq}
q:= \frac{\log(b/a)+\eta'}{\log(bc/ad)} >0,
\end{equation}
\item $N$ is taken sufficiently large such that $\frac{\log 2d}{\log N}< \frac{q}{2}$ and $s_0> 2-\frac{q}{2}$.
\end{enumerate}

The idea behind our construction is as follows. The parameters have been chosen in such a way so that in the projection $\pi^2_{\ii}F=F'$ the $y$ direction dominates, and cylinders deep in the construction of $F'$ typically have very thin width and long height. On the other hand, the typical behaviour for $F$ is affected by the sheer number of maps whose contraction in the $x$ and $y$ direction is determined by the matrix $A_1= \ldots = A_N$, and therefore typical projected cylinders in $\pi^2_{\ii} S_{\ii}F$ will be longer in the $x$ direction than the $y$ direction; in particular the orderings will not agree. Moreover, since the cylinders in $F'$ will be typically very tall and thin, typical covering boxes will not contain much of the projection $\pi_{\ii}^1F$, which will lead to a dimension drop.

\subsection{Typical covering boxes for the Bara\'nski carpet $F'$}

In this section we will solely consider the Bara\'nski system $\{T_1, T_2\}$ given by (\ref{baranski}). We will let $\mathcal{I}_2=\{1,2\}$ and $\mathcal{I}_2^{\ast}=\bigcup_{n \in \N} \mathcal{I}_2^n$ be the set of finite words as usual. For $\ii \in \mathcal{I}_2^*$, let $M_{\ii}$ denote the linear part of the map $T_{\ii}$. Throughout this section, for $\ii \in \mathcal{I}_2^*$, $\alpha_1(\ii)$ and $\alpha_2(\ii)$ will denote the first and second singular values of the matrix $M_{\ii}$. Let $\mu$ be the $(a^t,c^t)$ Bernoulli measure on $\Sigma_2=\{1,2\}^{\mathbb{N}}$. For $\ii \in \I_2^n$ let $[\ii]:=\{\jj \in \Sigma_2: \jj|_n =\ii\}$. For $i=1,2$, let $\Lambda_i(\mu)$ be the $i$th Lyapunov exponent of $\mu$ for the Bara\'nski system, which is defined via the sub-additive ergodic theorem as the negative constant such that for $\mu$ almost all $\ii \in \Sigma_2$,
$$\Lambda_i(\mu)= \lim_{n \to \infty} \frac{1}{n} \log \alpha_i(\ii|_n).$$

By \cite[Theorem 1.7]{fk} and assumption (i),
\begin{eqnarray*}
\Lambda_1(\mu)&=& a^t\log b+c^t \log d\\
\Lambda_2(\mu)&=&a^t\log a+c^t\log c.
\end{eqnarray*}
Define  $\lambda_1, \lambda_2 \in (0,1)$ by 
\begin{eqnarray*}
 \log \lambda_1=\log b^{a^t}d^{c^t}= \Lambda_1(\mu)\\
 \log \lambda_2=\log a^{a^t}c^{c^t}=\Lambda_2(\mu)
\end{eqnarray*} 
\vspace{4mm}
and let $f,g : \mathcal{I}_2^{\ast} \to (0,1)$ be defined by the identity
$$M_{\ii}= \begin{pmatrix} f(\ii) &0 \\0& g(\ii)\end{pmatrix}.$$

Given $\beta>1$ and a cover $\mathcal{C}_{\gamma}$ of $F'$ by squares of side $\gamma$, we say that a square $B \in \mathcal{C}_{\gamma}$ has the $\beta$-property if there exist at most 2 rectangles $B_1$ and $B_2$ of height at most $\gamma$ and width at most $\gamma^{\beta}$ such that $F' \cap (B \setminus B_1 \cup B_2)= \emptyset$. Given a cover $\mathcal{C}_{\gamma}$ of $F'$ by squares of side $\gamma$, we let $\mathcal{C}_{\gamma}^{\beta}$ denote the collection of squares in $\mathcal{C}_{\gamma}$ which have the $\beta$-property. The following lemma shows that for some $\beta>1$ most boxes in a cover of $F'$ have the $\beta$-property. 

\begin{lma} \label{baranski lma}
There exists $\kappa>0$ and $\beta>1$ such that for all $\gamma>0$ and any cover $\mathcal{C}_{\gamma}$ of $F'$ by disjoint squares of side $\gamma$,
$$\frac{\#\mathcal{C}_{\gamma}^{\beta}}{\#\mathcal{C}_{\gamma}} \geq 1-\gamma^{\kappa}.$$
\end{lma}

\begin{proof}
Write  $\tau= \frac{\log a-\log c}{\log d- \log b}$. Since $\frac{\log \lambda_2}{\log \lambda_1}>1$ we can fix $C>1$ sufficiently small so that 
$$\beta:= \frac{\log(\lambda_2 C^\tau)}{\log(\lambda_1/C)}>1.$$

Let $\gamma>0$ and fix a cover $\mathcal{C}_{\gamma}$ of $F'$ by disjoint squares of side $\gamma$.  Define
\[
\mathcal{J}(\gamma^{\beta})=\left\{\ii \in \mathcal{I}_2^*: f(\ii) \leq \gamma^{\beta} < f(\ii^-)\right\}
\]
 and note that for all $\ii \in \mathcal{J}(\gamma^{\beta})$, we have $\mu([\ii]) \approx \gamma^{\beta t}$ and $\beta \frac{\log \gamma}{\log a} \leq |\ii| \leq \beta \frac{\log \gamma}{\log c}$.  By Cram\'er's large deviations theorem for sums of i.i.d random variables \cite{cramer}, there exists $r \in (0,1)$ such that for all $n \geq 1$
\begin{eqnarray}
\mu\left(\left\{ \ii \in \Sigma_2 : \frac{\lambda_1^k}{g(\ii|_k)}> C^k  \textnormal{ for some $k \geq n$}\right\}\right) \leq r^n. \label{ld}
\end{eqnarray}
Define 
$$\mathcal{J}_{bad}(\gamma^{\beta})=\left\{\ii \in \mathcal{J}(\gamma^{\beta}): \frac{\lambda_1^{|\ii|}}{g(\ii)}>C^{|\ii|}\right\}.$$
By (\ref{ld}),
\[
\mu\left(\bigcup_{\ii \in \mathcal{J}_{bad}(\gamma^{\beta})}[\ii]\right) \leq r^{\frac{\beta \log \gamma}{\log a}}= \gamma^{\beta \frac{\log r}{\log a}}
\]
and let $\kappa= \beta \frac{\log r}{\log a}>0$. Since for all $\ii \in \mathcal{J}(\gamma^{\beta})$, we have $\mu([\ii]) \approx \gamma^{\beta t}$ and $\mu\left(\bigcup_{\ii \in \mathcal{J}(\gamma^{\beta})}[\ii]\right)=1$ it follows that
$$\frac{\# \mathcal{J}_{bad}(\gamma^{\beta})}{\# \mathcal{J}(\gamma^{\beta})} \lesssim \gamma^{\kappa}.$$
Let $\ii \in \mathcal{J}_{bad}(\gamma^{\beta})$,  write $n=|\ii|$ and let $n_1$ be the number of 1s that appear in the word $\ii$. Then, since $g(\ii)=b^{n_1}d^{n-n_1}< \frac{\lambda_1^n}{C^n}$, it follows that
$$n_1>n \frac{\log\left(\frac{\lambda_1}{Cd}\right)}{\log \left(\frac{b}{d}\right)}.$$
Therefore, recalling that $\lambda_1= b^{a^t}d^{c^t}$, $ \lambda_2=  a^{a^t}c^{c^t}$, and $a^t+c^t=1$,
\begin{eqnarray*}
f(\ii) = a^{n_1}c^{n-n_1}> \left(\frac{a}{c}\right)^{n \frac{\log\left(\frac{\lambda_1}{Cd}\right)}{\log \left(\frac{b}{d}\right)}} c^n  &=& a^{n(a^t+ \frac{\log(1/C)}{\log(b/d)})}c^{n(c^t+1-\frac{\log(b/Cd)}{\log(b/d)})}  \\ 
&=& \lambda_2^n C^{n(\frac{\log c-\log a}{\log b- \log d})}\\
 &=& \lambda_2^n C^{\tau n}\\
&=&(\lambda_2C^{\tau})^{|\ii|}.\end{eqnarray*}
In particular, $\gamma^{\beta} \gtrsim \lambda_2^n C^{\tau n}$ and so $|\ii|=n> \frac{\beta \log \gamma}{\log(\lambda_2 C^{\tau})}$.  This implies that if $\ii \in \mathcal{J}_{bad}(\gamma^{\beta})$ then 
\[
g(\ii)< \left(\frac{\lambda_1}{C}\right)^{\frac{\beta \log \gamma}{\log(\lambda_2 C^{\tau})}}=(\gamma^{\beta})^{\frac{\log(\lambda_1/C)}{\log(\lambda_2 C^{\tau})}}=\gamma^{\beta \cdot \beta^{-1}}=\gamma
\]
 and if $\ii \in \mathcal{J}(\gamma^{\beta}) \setminus \mathcal{J}_{bad}(\gamma^{\beta})$ then $g(\ii)\geq\gamma$. Therefore, given any $\ii \in \mathcal{J}_{bad}(\gamma^{\beta})$ and $\jj  \in \mathcal{J}(\gamma^{\beta}) \setminus \mathcal{J}_{bad}(\gamma^{\beta})$, the number of squares in $\mathcal{C}_{\gamma}$ that $T_{\ii}([0,1]^2)$ intersects can be at most 4 times the number of squares in $\mathcal{C}_{\gamma}$ that $T_{\jj}([0,1]^2)$ intersects. Therefore, denoting $\mathcal{C}_{\gamma}^{good}$ as the collection of squares in $\mathcal{C}_{\gamma}$ which do not intersect $T_{\ii}([0,1]^2)$ for any $ \ii \in \mathcal{J}_{bad}(\gamma^{\beta})$, we have
$$\frac{\# \mathcal{C}_{\gamma}^{good}}{\# \mathcal{C}_{\gamma}} \gtrsim 1-\gamma^{\kappa}.$$
In particular, since any square $B \in \mathcal{C}_{\gamma}^{good}$ can only intersect $T_{\ii}([0,1]^2)$ for $\ii \in \mathcal{J}(\gamma^{\beta})$ which necessarily has height $g(\ii) \geq \gamma$, it follows that there can be at most two distinct words $\ii, \jj \in \mathcal{J}(\gamma^{\beta})$ such that $B \cap T_{\ii}([0,1]^2) \neq \emptyset$ and $B \cap T_{\jj}([0,1]^2) \neq \emptyset$. By definition, $T_{\ii}([0,1]^2)$ and $T_{\jj}([0,1]^2)$ have widths $f(\ii), f(\jj ) \leq \gamma^{\beta}$. In particular, setting $B_1= B \cap T_{\ii}([0,1]^2)$ and $B_2=B \cap T_{\jj}([0,1]^2)$ we see that $B \in \mathcal{C}_{\gamma}^{\beta}$ and so
$$\frac{\# \mathcal{C}_{\gamma}^{\beta}}{\# \mathcal{C}_{\gamma}} \geq \frac{\# \mathcal{C}_{\gamma}^{good}}{\# \mathcal{C}_{\gamma}} \gtrsim 1-\gamma^{\kappa}$$
completing the proof.
\end{proof}

\subsection{Proof of dimension drop}

We now prove Theorem \ref{counter}. By using Lemma \ref{baranski lma} we show that for typical $\ii \in \mathcal{I}(\delta)$, most of the rectangles defined as `images under $S_{\ii}$' of covering boxes in the projection can be covered by only two boxes, rather than $\left(\frac{\alpha_1(\ii)}{\alpha_2(\ii)}\right)^t$ boxes as predicted by $\psi^s(\ii)$. By our assumptions on the defining parameters, this will imply that $\ubd F < s_0$. 

\begin{thm}[Refinement of Theorem \ref{counter}]
Let $F$ be the self-affine sponge that satisfies assumptions (i)-(iii) and additionally $N$ is taken large enough such that for all $\ii \in \mathcal{I}^{\ast}$,
\begin{eqnarray}
\left(\frac{\alpha_3(\ii)}{\alpha_2(\ii)}\right)^{\omega}=\left(\frac{N^{-|\ii|}}{\alpha_2(\ii)}\right)^{\omega} \leq \frac{\alpha_2(\ii)}{\alpha_1(\ii)} \label{omega}
\end{eqnarray}
for $\omega= \min\{\beta-1, \frac{\kappa}{t}\}$. Then $\ubd F < s_0$.
\end{thm}

\begin{proof}
Let $f', g' : \mathcal{I}^{\ast} \to (0,1)$ be defined by the identity
$$ A_{\ii}= \begin{pmatrix} f'(\ii) &0 &0\\0& g'(\ii)&0 \\ 0&0& N^{-|\ii|}\end{pmatrix}.$$
Define $\eta=\frac{\eta'}{\log N}$ and
$$\mathcal{I}_{good}(\delta)=\left\{\ii \in \mathcal{I}(\delta): \frac{g'(\ii)}{f'(\ii)} \leq N^{-|\ii|\eta}\right\}$$
and
$$\mathcal{I}_{bad}(\delta)=\{\ii \in \mathcal{I}(\delta): \frac{g'(\ii)}{f'(\ii)} >N^{-|\ii|\eta}\}.$$
In particular, notice that if $\ii \in \mathcal{I}_{good}(\delta)$ then $\frac{\alpha_2(\ii)}{\alpha_1(\ii)}=\frac{g'(\ii)}{f'(\ii)} \leq N^{-|\ii|\eta}$. Note that
\begin{eqnarray}
 N_{\delta}(F) \leq \sum_{\ii \in \mathcal{I}_{good}(\delta)} N_{\delta}(S_\ii(F))+  \sum_{\ii \in \mathcal{I}_{bad}(\delta)} N_{\delta}(S_\ii(F)). \label{2sums}
\end{eqnarray}
We begin by bounding the first sum in (\ref{2sums}). Fix $\ii \in \mathcal{I}_{good}(\delta)$. By taking $\gamma= \frac{\delta}{\alpha_2(\ii)}$ in Lemma \ref{baranski lma}, it follows that we can take a cover of $F' = S_{\ii}^{-1}\pi_{\ii}^2 S_{\ii}F$ by: $\lesssim \frac{\alpha_2(\ii)}{\delta}$ squares of side $\frac{\delta}{\alpha_2(\ii)}$ with the $\beta$-property and $\lesssim \left(\frac{\alpha_2(\ii)}{\delta}\right)\left(\frac{\delta}{\alpha_2(\ii)}\right)^{\kappa}$ squares of side $\frac{\delta}{\alpha_2(\ii)}$ without the $\beta$-property. First, consider the image under $S_{\ii}$ of a square with the $\beta$-property, which is a rectangle with sidelengths $\delta \frac{\alpha_1(\ii)}{\alpha_2(\ii)}$ and $\delta$. If $B_1$ and $B_2$ are the rectangles from the definition, then $S_{\ii}B_1$ and $S_{\ii}B_2$ are rectangles of height $\delta$ and width at most
$$\left(\frac{\delta}{\alpha_2(\ii)}\right)^{\beta} \alpha_1(\ii) \leq \frac{\delta}{\alpha_2(\ii)} \alpha_1(\ii) \frac{\alpha_2(\ii)}{\alpha_1(\ii)}= \delta,$$
where the first inequality follows by (\ref{omega}). Therefore $S_{\ii}B$ can be covered by 2 squares of sidelength $\delta$.

Next, consider a square $B$ without the $\beta$-property. As in Lemma \ref{keycover1}, $S_{\ii}B$ can be covered by $\lesssim \left(\frac{\alpha_1(\ii)}{\alpha_2(\ii)}\right)^t$ squares of side $\delta$.  Therefore, for arbitrary $s>s_0$, we can bound the first sum appearing in (\ref{2sums}) by
\begin{eqnarray}
&\,& \hspace{-15mm} \sum_{\ii \in \mathcal{I}_{good}(\delta)} N_{\delta}(S_\ii(F))  \nonumber \\ \nonumber \\
&\lesssim &  \sum_{\ii \in \mathcal{I}_{good}(\delta)} 2\alpha_2(\ii)\delta^{-1}+ \sum_{\ii \in \mathcal{I}_{good}(\delta)}\alpha_2(\ii)\delta^{-1} \left(\frac{\delta}{\alpha_2(\ii)}\right)^{\kappa} \left(\frac{\alpha_1(\ii)}{\alpha_2(\ii)}\right)^t \nonumber \\
&\lesssim & \delta^{-s}  \sum_{\ii \in \mathcal{I}_{good}(\delta)} \psi^s(\ii) \left(\frac{\alpha_2(\ii)}{\alpha_1(\ii)}\right)^t \nonumber \\ 
& \lesssim & \delta^{\eta-s} \sum_{\ii \in \mathcal{I}^{\ast}} \psi^s(\ii) \nonumber \\ \nonumber \\
&\lesssim& \delta^{\eta-s} \label{good} \end{eqnarray}
where the second line follows by (\ref{omega}).

Next we bound the second sum in (\ref{2sums}). Fix $\delta>0$ and notice that $\mathcal{I}(\delta)=\mathcal{I}^n$ for some $n$. In particular, $\ii \in \mathcal{I}_{bad}(\delta)$ if and only if the number $k$ of times that the symbol $(N+1)$ appears in the word $\ii$ satisfies
$$\frac{b^{n-k}d^k}{a^{n-k}c^k} \geq N^{-n \eta}$$
that is, $k \geq nq$, where $q$ is defined above in \eqref{defq}. Therefore, using Lemma \ref{keycover1} we obtain
\begin{eqnarray}
\sum_{\ii \in \mathcal{I}_{bad}(\delta)} N_{\delta}(S_{\ii}(F)) &\lesssim & \sum_{\ii \in \mathcal{I}_{bad}(\delta)} \alpha_1(\ii) \delta^{-1} \nonumber\\
&=& \delta^{-1} \sum_{k=qn}^n {n \choose k} N^{n-k} \max\{d^kb^{n-k}, c^k a^{n-k}\} \nonumber \\
&\leq& \delta^{-1} (2N^{1-q}d)^n \nonumber \\
&\approx& \delta^{-1} \delta^{-\frac{\log 2dN^{1-q}}{\log N}}. \label{bad}
\end{eqnarray}
By (\ref{2sums}), (\ref{good}), (\ref{bad}) and assumption (iii) we see that
$$\ubd F \leq \max\left\{s_0- \eta,  \, 2-q+\frac{\log 2d}{\log N}\right\}<s_0,$$
as required. 
\end{proof}

\section{A planar example exhibiting discontinuity of the pressure and dimension as functions of affinities} \label{discont}

Throughout this short section we always assume that the (families of) IFS that we consider all satisfy suitable separation conditions such as the strong open set condition.

 It is well-known that the dimension of a self-affine set need not depend continuously on the \emph{translational parts} of the defining affinity maps, when the linear parts are fixed. A classical example of this fact is the IFS  given by $\{T_1, T_2\}$ where
\begin{equation*}
T_i\begin{pmatrix}x\\y\end{pmatrix}=
\left\{ \begin{array}{cc}

\begin{pmatrix}1/2 &0\\0&1/3\end{pmatrix}\begin{pmatrix}x\\y\end{pmatrix}+\begin{pmatrix}0\\0\end{pmatrix} & i=1  \\
\begin{pmatrix}1/2&0\\0&1/3\end{pmatrix}\begin{pmatrix}x\\y\end{pmatrix}+\begin{pmatrix}\eps\\ 2/3\end{pmatrix} & i=2  
\end{array} \right.
\end{equation*}
for $\eps \in [0,1/2)$. The attractor of this system is easily seen to have box and Hausdorff dimension equal to 1 for $\eps>0$ and equal to $\log 2/\log 3$ for $\eps=0$. This discontinuity arises because the matrices defining the IFS have been fixed as generalised permutation matrices. In contrast, for generic choices of matrices, the dimension will be \emph{constant} under changes in the translations, and therefore will be (trivially) continuous in the translational parts. For example,  if the matrices are fixed to satisfy \cite[Theorem 1.1]{barany} then the Hausdorff and box dimensions are constant (thus continuous) in the translations.

Recently, there has been some interest in the regularity of the dimension as a function of the \emph{linear parts} of the defining affinity maps, where this time the translations are fixed. Feng and Shmerkin \cite{fengshmerkin} proved that the standard subadditive pressure function associated to matrix cocycles (and heavily used in the dimension theory of self-affine sets) is continuous in the matrices.  In particular this shows the affinity dimension varies continuously in the linear parts of the defining affinities, since it is defined as the zero of the pressure.  This answered a question of Falconer and Sloan who proved the result with some strong assumptions a few years earlier \cite{falconersloan}.  Morris gave an alternative proof of the result of Feng and Shmerkin in \cite{morris}, and recently it was also shown that in certain planar settings the affinity dimension is even analytic in the matrix coefficients of the defining affinities \cite{jm}.

As a consequence of \cite{fengshmerkin}, one can deduce that the dimension of a self-affine set is continuous in the linear parts of the defining affinities on `generic' parts of $\mathcal{GL}_2(\R)$. For example, the Hausdorff and box dimensions of the associated attractor are continuous in the linear parts of the affinities on subsets of $\mathcal{GL}_2(\R)$ which satisfy the assumptions of \cite[Theorem 1.1]{barany}.

In this section we provide a simple example of a planar self-affine system of the type studied in \cite{fraser} for which the box dimension of the attractor and associated pressure function are \emph{not} continuous in the defining matrices, for a fixed set of translations.  Note that the discontinuity is not caused by dimension drop in the projection.

  Let $a \in (1/3,1/2)$, $\eps \in [0,1/10]$ and consider the IFS 
$\{S_1, S_2,S_3, S_4, S_5\}$ where:
\begin{equation*}
S_i\begin{pmatrix}x\\y\end{pmatrix}=
\left\{ \begin{array}{ccc}

\begin{pmatrix}a&0\\0&1/3\end{pmatrix}\begin{pmatrix}x\\y\end{pmatrix}+\begin{pmatrix}0\\(i-1)/3\end{pmatrix} & i=1, 2, 3  \\
\begin{pmatrix}a&0\\0&1/3\end{pmatrix}\begin{pmatrix}x\\y\end{pmatrix}+\begin{pmatrix}1-a\\0\end{pmatrix} & i=4  \\
\begin{pmatrix}0&\eps\\\eps&0\end{pmatrix}\begin{pmatrix}x\\y\end{pmatrix}+\begin{pmatrix}9/10\\ 9/10\end{pmatrix}& i=5.\end{array} \right.
\end{equation*}

Recall from \cite{fraser} that the box dimension of the attractor of the above system is given as the root of the associated pressure function 
$$P(s):= \lim_{k \to \infty}\left(\sum_{\ii \in \mathcal{I}^k} \alpha_1(\ii)^{p_1(\ii)}\alpha_2(\ii)^{s-p_1(\ii)}\right)^{\frac{1}{k}}$$ where $\mathcal{I}=\{1,2,3,4,5\}$ and $p_1(\ii)$ is the box dimension of the projection of the attractor onto the one-dimensional subspace parallel to the longest side of the rectangle $T_{\ii}([0,1]^2)$.  We prove that the pressure of the above system is discontinuous at $\eps=0$, which immediately implies a discontinuity of the box dimension at $\eps=0$. We write $\mathcal{J} = \{1,2,3,4\}$.  First consider the situation where $\eps>0$.  Since the dimension of the projection onto the 2nd coordinate is clearly 1, and for $\eps>0$ the system is irreducible, the dimension of the projection onto the first coordinate is also 1 and therefore the pressure is given by
\begin{eqnarray*}
P(s)  \ = \  \lim_{k \to \infty} \left( \sum_{\ii \in \mathcal{I}^k} \alpha_1(\ii) \alpha_2(\ii)^{s-1} \right)^{1/k}  
 \ \geq \  \lim_{k \to \infty} \left( \sum_{\ii \in \mathcal{J}^k} a^k (1/3^k)^{s-1} \right)^{1/k} 
\ = \  4 a 3^{1-s}.
\end{eqnarray*}

Now, for $\eps=0$, the dimension of the projection of the attractor onto the 1st coordinate is
\[
p=\frac{\log 2}{\log (1/a)}<1
\]
and the pressure is 
\begin{eqnarray*}
P(s) \ = \ \lim_{k \to \infty} \left( \sum_{\ii \in \mathcal{I}^k} \alpha_1(\ii)^{p_1(\ii)} \alpha_2(\ii)^{s-p_1(\ii)} \right)^{1/k} 
&= & \lim_{k \to \infty} \left( \sum_{\ii \in \mathcal{J}^k} a^{kp} (1/3^k)^{s-p} \right)^{1/k} \\
&= & 4 a^p 3^{p-s} \\ 
&<& 4 a 3^{1-s}
\end{eqnarray*}
as required. 

\vspace{1cm}

\begin{centering}

\textbf{Acknowledgements}

JMF was financially supported by  an \emph{EPSRC Standard Grant} (EP/R015104/1). NJ was financially supported by a \emph{Leverhulme Trust Research Project Grant} (RPG-2016-194). The authors thank Ian Morris for suggesting the subsystem approach used in the proof of Lemma \ref{carries lemma} which allowed significant improvements to the exposition of the paper.
\end{centering}


\begin{thebibliography}{99}

\bibitem[B]{baranski}
K.~Bara\'nski.
Hausdorff dimension of the limit sets of some planar geometric constructions,
 {\em Adv. Math.}, {\bf 210}, (2007), 215--245.

\bibitem[BHR]{barany}
B. B\'ar\'any, M. Hochman and A. Rapaport.
Hausdorff dimension of planar self-affine sets and measures,
\emph{Invent. Math}, (to appear).


\bibitem[Be]{bedford}
T. Bedford.
 Crinkly curves, Markov partitions and box dimensions in self-similar sets,
 {\em Ph.D dissertation, University of Warwick}, (1984).
 


\bibitem[DS]{das}
T. Das and  D. Simmons.
The Hausdorff and dynamical dimensions of self-affine sponges: a dimension gap result, \emph{Invent. Math}, {\bf 210}, (2017),  85--134.


\bibitem[F1]{falconer1}
K.~J. Falconer.
 The Hausdorff dimension of self-affine fractals,
 {\em Math. Proc. Camb. Phil. Soc.}, {\bf 103}, (1988), 339--350.




\bibitem[F2]{falconer2}
K.~J. Falconer.
 The Hausdorff dimension of self-affine fractals II,
 {\em Math. Proc. Camb. Phil. Soc.}, {\bf 111}, (1992), 169--179.

\bibitem[F3]{falconer}
K. J. Falconer.
{\em Fractal Geometry: Mathematical Foundations and Applications},
 John Wiley \& Sons, Hoboken, NJ, 3rd. ed., 2014.


\bibitem[F4]{affinesurvey}
K.~J. Falconer.
Dimensions of Self-affine Sets - A Survey, \emph{Further Developments in Fractals and Related Fields}, Birkh\"auser, Boston, 2013, 115--134.

\bibitem[FK]{fk}
D.-J.  Feng and A. K\"{a}enm\"{a}ki. Equilibrium states of the pressure function for products of matrices, \emph{ Discrete  Cont. Dynam. Syst.}, {\bf  30}, (2011), 699--708.

\bibitem[FS1]{falconersloan}
K.~J. Falconer and A. Sloan.
Continuity of subadditive pressure for self-affine sets,
 {\em Real Anal. Ex.}, {\bf 34}, (2009), 413--428.




\bibitem[FS2]{fengshmerkin}
D.-J. Feng and P. Shmerkin. Non-conformal repellers and the continuity of pressure for matrix cocycles,
\emph{Geom. Func. Anal.}, {\bf 24}, (2014), 1101--1128.
  
\bibitem[FW]{fengaffine}
D.-J. Feng and Y. Wang. A class of self-affine sets and self-affine measures,
\emph{J. Fourier Anal. Appl.}, {\bf 11}, (2005), 107--124.


\bibitem[Fr]{fraser} J. M. Fraser.
On the packing dimension of box-like self-affine sets in the plane,
\emph{Nonlinearity}, \textbf{25}, (2012), 2075--2092.

\bibitem[FHOR]{fraserrobinson}
J.~M. Fraser, A. M. Henderson, E. J. Olson and J. C. Robinson.
On the Assouad dimension of self-similar sets with overlaps, 
  \emph{Adv. Math.}, {\bf 273}, (2015), 188--214.

\bibitem[FO]{FraserOrponen}
J.~M. Fraser and T. Orponen.
The Assouad dimensions of projections of planar sets,
 {\em Proc. Lond. Math. Soc.}, {\bf 114}, (2017), 374--398.

\bibitem[FY]{fraseryu}
J. M. Fraser and H. Yu.
Assouad type spectra for some fractal families,
\emph{ Indiana Univ.~Math.~J.},   {\bf 67}, (2018), 2005--2043.

\bibitem[GH]{Hare}
I. Garc\'{i}a and K. Hare.
Properties of Quasi-Assouad dimension,
(2017), arXiv:1703.02526v1



\bibitem[GL]{lalley-gatz}
D. Gatzouras and S.~P. Lalley.
Hausdorff and box dimensions of certain self-affine fractals,
 {\em Indiana Univ. Math. J.}, {\bf 41}, (1992), 533--568.
 
\bibitem[HR]{hochman}
M. Hochman and A. Rapaport. Hausdorff Dimension of Planar Self-Affine Sets and Measures with Overlaps. arXiv preprint arXiv:1904.09812 (2019).
 
 \bibitem[JM]{jm}
N. Jurga and I. Morris. 
Analyticity of the affinity dimension for planar iterated function systems with matrices which preserve a cone. arXiv preprint arXiv:1904.07699 (2019).

 \bibitem[K]{cramer}
A. Klenke. 
Probability theory: a comprehensive course. {\em Springer Science and Business Media}, 2013.

\bibitem[KP]{kenyon}
R. Kenyon and Y. Peres.
Measures of Full dimension on Affine-Invariant Sets,
\emph{Erg. Th. and Dyn. Syst.} {\bf 16}, (1996), 307--323.

\bibitem[LX]{LuXi}
F. L\"u and L. Xi.
Quasi-Assouad dimension of fractals,
\emph{J. Fractal Geom.}, {\bf 3}, (2016), 187--215.

\bibitem[Mc]{mcmullen}
C. McMullen.
 The Hausdorff dimension of general Sierpi\'nski carpets,
 {\em Nagoya Math. J.}, {\bf 96}, (1984), 1--9.

\bibitem[M]{morris}
I. Morris. An inequality for the matrix pressure function and applications,
\emph{Adv. Math.}, {\bf 302}, (2016), 280--308.

\bibitem[PS]{problems_update}
Y. Peres and B. Solomyak
Problems on self-similar sets and self-affine sets: an update, \emph{Fractal geometry and stochastics, II} (Greifswald/Koserow, 1998), 95--106,
\emph{Progr. Probab.}, 46, Birkh\"auser, Basel, 2000. 

\bibitem[S]{solomyak}
B. Solomyak. Measure and dimension for some fractal families. \emph{Mathematical Proceedings of the Cambridge Philosophical Society}. Vol. 124. No. 3. Cambridge University Press, 1998.

\end{thebibliography}
\end{document}